\documentclass[a4paper]{article}

\usepackage[utf8]{inputenc} 
\usepackage[T1]{fontenc} 
\usepackage{amsmath,amsthm,latexsym} 
\usepackage{amssymb,amsfonts} 
\usepackage{algpseudocode} 
\usepackage{algorithm} 
\usepackage{url}
\usepackage{graphicx}
\usepackage{xcolor}
\usepackage[margin=1.25in]{geometry} 

\graphicspath{{Fig/}}

\theoremstyle{plain}%
\newtheorem{theorem}{Theorem}
\newtheorem{assumption}[theorem]{Assumption}%
\newtheorem{lemma}[theorem]{Lemma} 
\newtheorem{remark}{Remark}%

\numberwithin{equation}{section}

\usepackage{subcaption} 

\newcommand{\R}{\mathbb{R}} 
\DeclareMathOperator*{\argmin}{arg\,min} 
\newcommand{\grad}{\nabla} 
\renewcommand{\t}[1]{\widetilde{#1}} 
\newcommand{\bigO}{\mathcal{O}} 
\newcommand{\thmref}[1]{Theorem~\ref{#1}}

\newcommand{\secref}[1]{Section~\ref{#1}}
\newcommand{\figref}[1]{Figure~\ref{#1}}
\newcommand{\tabref}[1]{Table~\ref{#1}}
\newcommand{\lemref}[1]{Lemma~\ref{#1}}

\newcommand{\revision}[1]{#1}
\newcommand{\rerevision}[1]{#1}

\begin{document}

\title{Analyzing Inexact Hypergradients for Bilevel Learning}

\author{
Matthias J. Ehrhardt\thanks{Department of Mathematical Sciences, University of Bath, Bath, BA2 7AY, UK (\texttt{m.ehrhardt@bath.ac.uk})}
\and 
Lindon Roberts\thanks{School of Mathematics and Statistics, University of Sydney, Camperdown NSW 2006, Australia (\texttt{lindon.roberts@sydney.edu.au}).}}

\maketitle

\begin{abstract}
Estimating hyperparameters has been a long-standing problem in machine learning. We consider the case where the task at hand is modeled as the solution to an optimization problem. Here the exact gradient with respect to the hyperparameters cannot be feasibly computed and approximate strategies are required. 
We introduce a unified framework for computing hypergradients that generalizes existing methods based on the implicit function theorem and automatic differentiation/backpropagation, showing that these two seemingly disparate approaches are actually tightly connected.
Our framework is extremely flexible, allowing its subproblems to be solved with any suitable method, to any degree of accuracy.
We derive a priori and computable a posteriori error bounds for all our methods, and numerically show that our a posteriori bounds are usually more accurate.
Our numerical results also show that, surprisingly, for efficient bilevel optimization, the choice of hypergradient algorithm is at least as important as the choice of lower-level solver.
\end{abstract}

\\

\section{Introduction}
In this work we consider the hyperparameter tuning problem framed as a bilevel optimization problem \cite{Dempe2015bilevel, DelosReyes2021, Crockett2021bilevel} where we aim to solve
\begin{subequations} \label{eq_general_problem}
\begin{align}
	&\min_{\theta\in\R^n} F(\theta) := \frac 1 m \sum_{i=1}^{m} f_i(x_i^*(\theta)) + r(\theta) \label{eq_upper_level} \\
	\text{s.t.} \quad &x_i^*(\theta) := \argmin_{x\in\R^d} g_i(x, \theta),  \qquad  \forall i=1,\ldots,m, \label{eq_x_theta}
\end{align}
\end{subequations}
where the lower-level functions $g_i$ are smooth in $x$ and $\theta$ and strongly convex in $x$, and the upper-level functions $f_i$ and $r$ are smooth but possibly nonconvex.
Our main motivation for studying \eqref{eq_general_problem} is the problem of supervised bilevel learning, where we may have $f_i(x) = \|x-x^\dagger_i\|^2$ for example, where $x^\dagger_i$ is the desired outcome of the lower-level problem \eqref{eq_x_theta}.

Problems of the form~\eqref{eq_general_problem} are ubiquitous in every aspect of science and tasks such as clustering, time series analysis and image reconstruction can be modeled as such. As a simple example \rerevision{with $n=1$}, we could \rerevision{think of} $\theta \geq 0$ as a choice of regularization weight suitable for a family of regularized regression problems $g_i$, i.e. $g_i(x, \theta) = \frac12 \|A x - y_i\|^2 + \theta R(x)$ \rerevision{for $i=1,\ldots,m$}. Similarly one can select a regularizer~\cite{Kunisch2013bilevel} or noise model~\cite{DeLosReyes2013}. Other models use many more parameters \rerevision{like} an input-convex neural \rerevision{network} as a regularizer \cite{Amos2017inputconvex, Mukherjee2020inputconvex} or \rerevision{the} sampling of the forward operator for image compression \cite{Hoeltgen2013compression, Chen2014} or MRI~\cite{Sherry2020sampling}.

When the number of parameters \rerevision{$n$} is small, the problem \eqref{eq_general_problem} can be efficiently solved by search methods (e.g.\revision{,} \cite{McKay1979, Bergstra2012}) or derivative-free approaches, see\revision{, e.g.,} \cite{Hutter2011, Snoek2012} for general hyperparameter search and \cite{Ehrhardt2021} for bilevel learning.

Motivated by bilevel learning, we are interested in algorithms which can scale to millions of parameters (or more), and so consider \eqref{eq_general_problem} with first-order methods used for solving the lower-level problem \eqref{eq_x_theta} and upper-level problem \eqref{eq_upper_level}.
If the \rerevision{functions $f_i$ and $g_i$ in \eqref{eq_upper_level} and \eqref{eq_x_theta} respectively are} sufficiently smooth, it is well-known (see\revision{, e.g.,} \cite{Bengio2000}) that \rerevision{the} gradients of the upper-level objective \rerevision{\eqref{eq_upper_level}} (also called hypergradients) can be computed via 
\begin{align}
    \grad (f_i \circ x_i^*)(\theta) = \partial x_i^*(\theta)^T \grad f_i(x_i^*(\theta)), \qquad i=1,\ldots,m, \label{eq_desired_grad_calc}
\end{align}
where $\partial x_i^*(\theta)$ is the derivative of the minimizer $x_i^*(\theta)$ with respect to the parameters $\theta$.
However, in the context of large-scale problems and general lower-level objectives $g_i$, access to the true lower-level minimizers $x_i^*(\theta)$ is unreasonable and so \eqref{eq_desired_grad_calc} cannot be evaluated.

The primary purpose of this work is to develop methods for efficiently evaluating $\grad (f_i \circ x_i^*)$. 
Such methods compute an approximate lower-level solution to be used in some form to then compute an approximation to the hypergradient~\eqref{eq_desired_grad_calc}. 
While this approach has been often used successfully in practice, the interplay between the accuracy \revision{of the computed approximation to $x_i^*(\theta)$} and its impact on the hypergradients are neither fully understood nor fully utilized.

In this work we introduce a unified framework for computing approximate hypergradients, which come with numerous concrete implementations and several error bounds suitable for different purposes.
Our main theoretical results are:
\begin{itemize}
    \item Showing that two promising hypergradient estimation algorithms, based on the \revision{implicit} function theorem \cite{Pedregosa2016,Zucchet2022} and inexact automatic differentiation (AD)/backpropagation \cite{Mehmood2020} respectively, are essentially the same underlying algorithm. This surprising result allows us to unify two very disparate hypergradient estimation methodologies into a general framework.
    \item Deriving general error bounds for the underlying general hypergradient estimation framework. Our error bounds are both a priori, based on known convergence rates of the constituent algorithms, and a posteriori, yielding computable error estimates suitable for use inside a bilevel optimization framework.
\end{itemize}
Our numerical results then compare the accuracy and efficiency of different hypergradient estimation techniques, as well as studying the impact of hypergradient algorithms on the performance of bilevel optimization routines.
Importantly, and perhaps surprisingly, our numerical evidence shows that, in the context of bilevel optimization, \emph{the choice of hypergradient algorithm is at least as important as the choice of lower-level solver}.

\subsection{Existing \rerevision{work}}
An explicit form for the derivative of the minimizer of a smooth, strongly convex \rerevision{optimization problem} is given by the classical \revision{implicit} function theorem (e.g.\revision{,}~\cite{Bengio2000}).
This formulation, where the resulting linear system \rerevision{of equations} is solved inexactly using the conjugate gradient (CG) method, was used as the basis of the Hyperparameter Optimization with Approximate Gradient (HOAG) algorithm for bilevel optimization in \cite{Pedregosa2016}.
The analysis of the underlying hypergradient estimation was refined in the more recent work \cite{Zucchet2022}.
Our unified approach is fundamentally based on these ideas, but we demonstrate how this approach also incorporates AD-based methods and have a more refined error analysis.

The other hypergradient estimation technique we consider is reverse-mode \rerevision{automatic differentiation}, also known as backpropagation, of a lower-level iterative solver.
Originally, AD for iterative methods was first applied to fixed point iterations \cite{Christianson1994}.
This was extended more recently in \cite{Mehmood2020} to parametric strongly convex optimization using gradient descent and heavy ball momentum.
Here, the authors prove \revision{\rerevision{sub-optimal} linear} convergence rates of hypergradients using AD, and introduce an `inexact' AD method with improved linear convergence rates.
Our analysis unifies the accelerated `inexact' approach from \cite{Mehmood2020} with the traditional analysis from \cite{Pedregosa2016,Zucchet2022}, introduces a new family of a posteriori bounds, and carefully considers the numerical performance of these options.

Both these perspectives (implicit function theorem and backpropagation) were considered separately in \cite{Grazzi2020a} where the lower-level problem \eqref{eq_x_theta} is replaced with a fixed-point iteration, and (separate) a priori bounds are derived in both cases.
Our work is similar in considering both techniques, but we give a unified perspective showing how they can be seen as the same method, and hence unified results can be shown.
We also extend the analysis by providing computable a posteriori bounds, and showing that these are often tighter in practice.

More generally, several approaches exist for solving the general bilevel problem \eqref{eq_general_problem} (instead of considering specifically the estimation of hypergradients).
This typically involves \revision{alternating between} a given number of iterations of a solver for the lower-level problem \eqref{eq_x_theta} with a (possibly different) number of iterations of an upper-level solver, both typically first-order methods.
In \cite{Ghadimi2018}, the number of lower-level gradient descent iterations is pre-specified. In \cite{Ji2021}\revision{, the} authors use a constant number of lower-level iterations for each upper-level iteration. Alternatively, \cite{Hong2020} uses just one iteration but with different upper- vs.~lower-level stepsizes.

When the lower-level problem is replaced by a finite algorithm, then gradients can be computed exactly, see\revision{, e.g.,} \cite{Ochs2015, Maclaurin2015, Shaban2019}. This is not the case here as the lower-level solution can only be approximated by the limit of an algorithm and thus requires special care.


%
%


In \cite{Suonpera2022} the authors propose an alternative optimality system including the upper-level unknown, the lower-level unknown and the adjoint state which corresponds to the derivative of the lower-level unknown with respect to the upper-level parameters.

Both the upper- and lower-level problem can be extended to stochastic optimization problems with only stochastic gradients available, see\revision{, e.g.,} \cite{Grazzi2020, Ji2021}. While this is clearly a challenging research question itself, it is fairly independent of the challenge considered in the present paper.

\subsection{Contributions}
In this work we introduce a unified perspective on hypergradient estimation, incorporating the different techniques analyzed in \cite{Pedregosa2016,Zucchet2022,Grazzi2020a,Mehmood2020}.
This includes estimates of the hypergradient based on the implicit function theorem, with inexact linear solves using iterative methods such as CG, and gradient descent and heavy ball variants of the inexact AD method from \cite{Mehmood2020} (which converge faster than standard AD).
We show that all of these approaches correspond to the same basic procedure: approximately solve the lower-level problem and then approximately solve the resulting implicit function theorem linear system.
As far as we are aware, showing that inexact AD is the same as using the implicit function theorem is a new result showing a surprising connection between analytic and symbolic gradient estimation.

With our resulting unified method, we prove a priori bounds similar to those in \cite{Pedregosa2016,Zucchet2022,Grazzi2020a}, where first-order methods applied to both subproblems yields a linear convergence rate.
Our analysis is extremely general: it enables both subproblems to be solved with any suitable algorithm, and run for any number of iterations or up to any stopping tolerance.

We then extend our analysis to prove a posteriori bounds on our flexible hypergradient estimators.
Here we construct explicit and computable error bounds on the hypergradient.
These bounds are potentially useful if we wish to apply existing algorithms for nonconvex optimization with inexact gradients to solve~\eqref{eq_general_problem}, for example  frameworks such as \cite{Berahas2021,Cao2022}.

We then study different variants of our hypergradient algorithm numerically.
First we use a simple linear least-squares problem to demonstrate the correctness of our bounds, and show that our new a posteriori analysis typically provides stronger bounds than the more common a priori analysis (as well as being computable in practice).
We then consider a data hypercleaning problem, comparing different combinations of algorithms for the lower-level solver and the implicit function theorem linear solver.
Unsurprisingly, the better the algorithms used (e.g.\revision{,}~heavy ball rather than gradient descent), the faster the optimization progresses.
However perhaps surprisingly, we show that the choice of hypergradient algorithm is at least as important as the choice of lower-level solver for the purposes of efficient bilevel optimization, demonstrating the importance of the hypergradient estimation problem.
Lastly, we conclude by applying our method to the bilevel problem of \rerevision{finding} a neural net regularizer for image denoising, and demonstrate that the \rerevision{obtained} regularizer can outperform the \revision{standard total variation regularizer}. 

The paper is structured as follows: in \secref{sec_background} we summarize the existing \revision{implicit} function theorem and inexact backpropagation theory. In \secref{sec_unified_alg} we prove that inexact backpropagation is just a special case of the \revision{implicit} function theorem \rerevision{approach} which motivates our unified hypergradient framework.
We prove the a priori and a posteriori error bounds for our approach in \secref{sec_bounds} and give numerical results in \secref{sec_numerics}.

\subsection{Notation}


Throughout, we will use $\|\cdot\|$ to be the Euclidean norm of vectors and operator 2-norm of matrices.
Both partial and total derivatives are denoted by $\partial$. If the function we take the derivative of is scalar-valued, then we denote its derivative by $\nabla$ and refer to it as the gradient. 
Specifically, given \eqref{eq_general_problem} we have
\begin{itemize}
    \item $\partial_y g_i:\R^d\times \R^n \to \R^d$ and $\partial_{yy} g_i:\R^d\times \R^n \to \R^{d\times d}$ are the gradient and Hessian (respectively) of $g_i(y,\theta)$ with respect to $y$ (for fixed $\theta$). That is, $[\partial_y g_i(y,\theta)]_j := \frac{\partial g_i(y,\theta)}{\partial y_j}$ and $[\partial_{yy} g_i(y,\theta)]_{j,k} := \frac{\partial^2 g_i(y,\theta)}{\partial y_j \partial y_k}$;
    \item $\partial_{y}\partial_{\theta} g_i:\R^d\times \R^n\to \R^{d\times n}$ is the Jacobian of $\partial_y g_i(y,\theta)$ with respect to $\theta$, where $[\partial_{y}\partial_{\theta} g_i(y,\theta)]_{j,k} :=\frac{\partial^2 g_i(y,\theta)}{\partial y_j \partial \theta_k}$;
    \item $\partial x_i^*:\R^n\to\R^{d\times n}$ is the derivative of $x_i^*(\theta)$ with respect to $\theta$, with $[\partial x_i^*(\theta)]_{j,k} := \frac{\partial [x_i^*(\theta)]_j}{\partial \theta_k}$; 
    \item $\grad f_i:\R^d\to\R^d$ is the gradient of $f_i(x)$ with respect to $x$ (or $\theta$ for $\grad F$ and $\grad r$).
\end{itemize}
At various points throughout we will drop the indexing on $i$ and explicit dependencies on $\theta$ for simplicity of presentation.

\revision{Note that this framework also handles complex-valued inputs and operators (with real-valued objectives to ensure optimization over an ordered field), by treating the real and imagingary parts of any complex quantities as two real-valued quantities.}

\section{Background} \label{sec_background}
We begin by outlining three existing approaches for calculating the hypergradient
\begin{align}
    \grad F(\theta) = \frac{1}{m} \sum_{i=1}^{m} \grad (f_i \circ x_i^*)(\theta) + \grad r(\theta) = \frac{1}{m}\sum_{i=1}^{m} \partial x_i^*(\theta)^T \grad f_i(x_i^*(\theta)) + \grad r(\theta) \label{eq_hypergradient},
\end{align}
where the second equality follows from the chain rule.\footnote{Our analysis is also suitable for the more general case where $f_i$ has an explicit dependence on $\theta$, i.e.~$f_i(x_i^*(\theta), \theta)$, but we use our simpler formulation for ease of presentation.}
In particular, we are concerned with how we may accurately compute \revision{$\grad F(\theta)$} in the setting where $x_i^*(\theta)$ cannot be exactly determined.
For $\partial x_i^*$ to be well-defined, we require the following assumptions on $g_i$.

\begin{assumption} \label{ass_smoothness}
	For each $i=1,\ldots,m$, the lower-level objective $g_i$ is twice continuously differentiable in $y$ and there exists $0<\mu_i(\theta)\leq L_i(\theta)$ such that $\mu_i(\theta) I \preceq \partial_{yy} g_i(y,\theta) \preceq L_i(\theta) I$ for all $y$\revision{, where $A\preceq B$ means $B-A$ is positive semidefinite}.
	Furthermore, each of $g_i$, $\partial_y g_i$ and $\partial_{yy} g_i$ are continuous in $\theta$.
\end{assumption}

\revision{In practice, Assumption~\ref{ass_smoothness} may require the domain of $\theta$ to be restricted, which \rerevision{means} more careful selection of the upper-level solver \rerevision{is required}, but \rerevision{this} does not affect the results below regarding the inexact computation of $\grad F(\theta)$.}

It is well known (see\revision{, e.g.,} \cite{Bengio2000, Ehrhardt2021}) that under Assumption~\ref{ass_smoothness}, the map $\theta \mapsto x_i^*(\theta)$ is continuously differentiable with derivative $\partial x_i^*(\theta) = D_i(x_i^*(\theta),\theta)^T$. This formulation relies on the function
\begin{align}
    D_i(y,\theta) := -B_i(y,\theta)^T A_i(y,\theta)^{-1} \in \R^{n\times d} \label{eq_ift_basic}
\end{align}
and associated matrices 
\begin{align}
    A_i(y,\theta) := \partial_{yy} g_i(y,\theta) \in \R^{d\times d} \qquad \text{and} \qquad B_i(y,\theta) := \partial_y \partial_{\theta} g_i(y,\theta) \in \R^{d\times n}. 
\end{align}
Note that Assumption~\ref{ass_smoothness} implies that $A_i(y, \theta)$ is invertible and thus $D_i$ is well-defined. Moreover, for all $y$ and $\theta$ it holds that
\begin{align} \label{eq:bound:Ainverse}
\|A_i(y, \theta)^{-1}\| \leq \mu_i(\theta)^{-1}.
\end{align}

\begin{remark} \label{rem_drop_theta}
From here, we drop the indexation by $i$ and the explicit dependence on $\theta$ (since $\theta$ does not vary throughout) for simplicity of notation. That is, we will now write $A(y)$, $x^*$ and $\mu$ instead of $A_i(y,\theta)$, $x_i^*(\theta)$ and $\mu_i(\theta)$, for example.
\end{remark}

Inserting \eqref{eq_ift_basic} into \eqref{eq_hypergradient} and ignoring the average and $r$ for now, leads to another formulation of the hypergradient 
\begin{align}
    h^* := - B(x^*)^T A(x^*)^{-1} \grad f(x^*) = D(x^*) \grad f(x^*). \label{eq_ift}
\end{align}
This formulation of the hypergradient is not practical: it is too costly to compute to high accuracy for most relevant applications since $x^*$ has to be computed iteratively (via a suitable strongly convex solver) and a system of linear equations the size of the number of lower-level unknowns (often exceeding millions in imaging applicatons) has to be solved. Therefore, we investigate methods which acknowledge \revision{the} fact that both computations cannot or should not be carried out accurately. Particularly, we focus on methods which are independent of the specific \revision{algorithm} for solving the lower-level problem \eqref{eq_x_theta}, and so instead we assume that we have an approximate minimizer $x_{\varepsilon}$ of \eqref{eq_x_theta}, such that
\begin{align}
    \|x_{\varepsilon} - x^*\| \leq \varepsilon. \label{eq_lower_inexact_solve}
\end{align}
Finally, we will use the following smoothness assumptions for the results below.

\begin{assumption} \label{ass_smoothness_extra}
$A$ (resp.~$B$) is $L_A$- (resp.~$L_B$-) Lipschitz continuous in $y$, uniformly for all $\theta$.
\end{assumption}

\begin{assumption} \label{ass_smoothness_upper}
$\nabla f$ is Lipschitz continuous with constant $L_{\nabla f}$.
\end{assumption}

\subsection{\revision{Implicit} Function Theorem Approach}


Given \eqref{eq_lower_inexact_solve} and \eqref{eq_ift}, a natural approximation is 
\begin{align}
    h^* \approx h_{\varepsilon} := -B(x_{\varepsilon})^T A(x_{\varepsilon})^{-1} \grad f(x_{\varepsilon}).
\end{align}
For large-scale problems, full matrix inversion/linear solves is impractical and so iterative methods are preferred.
Given that $A(x_{\varepsilon})$ is symmetric positive definite by Assumption~\ref{ass_smoothness}, we consider using the conjugate gradient method (CG).
The symmetry of $A$ gives two possible approaches:
\begin{itemize}
    \item Approximately solve $A(x_{\varepsilon})^{-1} \grad f(x_{\varepsilon})$ and pre-multiply by $-B(x_{\varepsilon})^T$
    \item Approximately solve $B(x_{\varepsilon})^T A(x_{\varepsilon})^{-1} = [A(x_{\varepsilon})^{-1} B(x_{\varepsilon})]^T $ and post-multiply by $\grad f(x_{\varepsilon})$.
\end{itemize}
Of these, the former is preferable as the linear solve has only one right-hand side, reducing the total number of matrix-vector products.
This leads to a natural algorithm to compute the hypergradient, given in Algorithm~\ref{alg:ift+cg}, which we refer to as \revision{IFT+CG}, relating to its association with the implicit function theorem.

\begin{algorithm}[tb]
\caption{IFT+CG to compute gradient estimate $h_{\varepsilon, \delta}$}\label{alg:ift+cg}
\begin{algorithmic}[1]
\Require tolerances $\varepsilon,\delta>0$
\State Find an approximate solution $x_{\varepsilon}$ satisfying \eqref{eq_lower_inexact_solve}.
\State Using CG, find $q_{\varepsilon,\delta}$ satisfying
    \begin{align}
        \|A(x_{\varepsilon}) q_{\varepsilon,\delta} - \grad f(x_{\varepsilon})\| \leq \delta. \label{eq_H_system2_resid}
    \end{align}
\State Compute gradient estimate $h_{\varepsilon, \delta} := -B(x_{\varepsilon})^T q_{\varepsilon,\delta}$.
\end{algorithmic}
\end{algorithm}


The first analysis of IFT+CG in the context of bilevel optimization was in \cite{Pedregosa2016}, which gives the following.
\begin{theorem}[Theorem 1, \cite{Pedregosa2016}] \label{thm_ift_v1}
    Suppose Assumptions~\ref{ass_smoothness}, \ref{ass_smoothness_extra} and \ref{ass_smoothness_upper} hold, and  $\varepsilon$ is sufficiently small\footnote{Specifically, the proof requires that $\sum_k \varepsilon_k < \infty$ and we take $k$ sufficiently large, where $\varepsilon_k$ is the value of $\varepsilon$ chosen in iteration $k$ of the bilevel optimization.}. 
    Then the error produced by IFT+CG with $\delta=\varepsilon$ satisfies $\|h_{\varepsilon,\varepsilon}-h^*\| = \bigO(\varepsilon)$.
\end{theorem}

This result was strengthened in the recent work \cite{Zucchet2022}, which considers IFT but with any method for solving \eqref{eq_H_system2_resid}, not just CG.

\begin{theorem}[Theorem 3.1.2, \cite{Zucchet2022}] \label{thm_ift_v2}
    Suppose Assumptions~\ref{ass_smoothness}, \ref{ass_smoothness_extra} and \ref{ass_smoothness_upper} hold, and $\varepsilon < \frac{\mu}{2\max\revision{\{L_A,L_B\}}}$.
    Then the error produced by IFT+CG \revision{satisfies} $\|h_{\varepsilon,\delta}-h^*\| \leq C(\varepsilon+\delta)$ for some constant $C$. 
\end{theorem}

Our analysis in \secref{sec_bounds} improves on Theorems~\ref{thm_ift_v1} and~\ref{thm_ift_v2} by removing the restriction on $\varepsilon$ and making all constants explicit.

\subsection{Inexact Automatic Differentiation}
We now consider an alternative approach for hypergradient estimation based on automatic differentiation (backpropagation), as developed in \cite{Mehmood2020}.

The motivation is to consider solving the lower-level problem to find $x_{\varepsilon}$ \eqref{eq_lower_inexact_solve} by applying a first-order method to \eqref{eq_x_theta} such as gradient descent
\begin{align}
	x^{(k+1)} = x^{(k)} - \alpha \nabla g(x^{(k)}), \label{eq_x_gd}
\end{align}
or Polyak's heavy ball momentum
\begin{align}
	x^{(k+1)} = x^{(k)} - \alpha \nabla g(x^{(k)}) + \beta (x^{(k)} - x^{(k-1)}). \label{eq_x_hb}
\end{align}
If we suppose that $K$ iterations of either \eqref{eq_x_gd} or \eqref{eq_x_hb} are run from a fixed starting point $x^{(0)}$, then reverse-mode automatic differentiation (AD) with respect to $\theta$ applied to \eqref{eq_x_hb} gives
: initialize $\t{x}^{(K)} := \grad f(x^{(K)})$, $\t{x}^{(K+1)}:=0\in\R^d$ and $h^{(0)}:=0\in\R^n$, then iterate
\begin{subequations} \label{eq_hb_exact_ad}
\begin{align}
    h^{(k+1)} &= h^{(k)} - \alpha B(x^{(K-k-1)})^T \t{x}^{(K-k)}, \\
    \t{x}^{(K-k-1)} &= \t{x}^{(K-k)} - \alpha A(x^{(K-k-1)}) \t{x}^{(K-k)} + \beta(\t{x}^{(K-k)} - \t{x}^{(K-k+1)}),
\end{align}
\end{subequations}
for $k=0,\ldots,K-1$. 
For gradient descent \eqref{eq_x_gd}, the same iteration \eqref{eq_hb_exact_ad} holds but with $\beta=0$.
In both cases, the final gradient estimator is $h^{(K)}$.

The key insight of \cite{Mehmood2020} is that \eqref{eq_hb_exact_ad} can be made to converge faster\footnote{Linearly \revision{with an improved rate}.} (and with fewer Hessian/Jacobian evaluations) by replacing $B(x^{(K-k)-1})$ and $A(x^{(K-k)-1})$ with the final Jacobian and Hessian, $B(x^{(K)})$ and $A(x^{(K)})$ for all $k$.
This motivates the inexact automatic differentiation (IAD) methods IAD+GD and IAD+HB given in Algorithms~\ref{alg:iad+gd} and~\ref{alg:iad+hb} respectively.

\begin{algorithm}[tb]
\caption{IAD+GD (fixed $K$) to compute gradient estimate $h^{(K)}$}\label{alg:iad+gd}
\begin{algorithmic}[1]
\Require iteration count $K$ and stepsize $\alpha$ \Comment{\revision{(note: this represents a code comment)}}
\State run $K$ iterations of \eqref{eq_x_gd} to get $x^{(K)} \approx x^*$. 
\State initialize  $\t{x}^{(0)} := \grad f(x^{(K)})$ and $h^{(0)} = 0 \in \R^{n}$
\State \textbf{for} $k=0,\ldots,K-1$ \Comment{$K$ iterations of \emph{inexact} reverse-mode AD}
    \begin{subequations} \label{eq_gd_inexact_ad}
    \begin{align}
	    h^{(k+1)} &= h^{(k)} - \alpha B(x^{(K)})^T \t{x}^{(k)}, \\
	    \t{x}^{(k+1)} &= \t{x}^{(k)} - \alpha A(x^{(K)}) \t{x}^{(k)}.
    \end{align}
\end{subequations}
\end{algorithmic}
\end{algorithm}
%
%
\begin{algorithm}[h!]
\caption{IAD+HB (fixed $K$) to compute gradient estimate $h^{(K)}$}\label{alg:iad+hb}
\begin{algorithmic}[1]
\Require iteration count $K$\revision{,} stepsize $\alpha$\revision{,} and momentum parameter $\beta$
\State run $K$ iterations of \eqref{eq_x_hb} to get $x^{(K)} \approx x^*$.
\State initialize  $\t{x}^{(0)} := \grad f(x^{(K)})$, $\t{x}^{(-1)}=0\in\R^d$ and $h^{(0)} = 0 \in \R^{n}$
\State \textbf{for} $k=0,\ldots,K-1$ \Comment{$K$ iterations of \emph{inexact} reverse-mode AD}
\begin{subequations} \label{eq_hb_inexact_ad}
    \begin{align}
	    h^{({k+1})} &= h^{(k)} - \alpha B(x^{(K)})^T \t{x}^{(k)}, \\
	    \t{x}^{(k+1)} &= \t{x}^{(k)} - \alpha A(x^{(K)}) \t{x}^{(k)} + \beta(\t{x}^{(k)} - \t{x}^{(k-1)}). \label{eq_hb_inexact_ad_2}
    \end{align}
\end{subequations}
\end{algorithmic}
\end{algorithm}

The main result from \cite{Mehmood2020} is the following.
\begin{theorem}[Propositions 10 \& 17, \cite{Mehmood2020}] \label{thm_inexact_ad_orig}
    Suppose Assumptions~\ref{ass_smoothness} and \ref{ass_smoothness_extra} hold. Further assume that $\|B(y)\| \leq B_{\max}$ for each $y$.\footnote{Technically, \cite{Mehmood2020} also requires that $B_{\max}$ does not depend on $\theta$ and that the bound holds uniformly for all $\theta$.
    }
    Then:
    \begin{itemize}
        \item If $\alpha \leq 1/L$, then there exists $\lambda_{GD}\in[0,1)$ such that IAD+GD (fixed $K$) gives a hypergradient estimate $h^{(K)}$ satisfying
        \begin{align}
            \|h^{(K)} - D(x^{(K)}) \grad f(x^{(K)})\| \leq \lambda_{GD}^K \frac{B_{\max}}{\mu} \|\grad f(x^{(K)})\|.
        \end{align}
        The optimal rate\footnote{\revision{Here, $L$ and $\mu$ denote $L(\theta)$ and $\mu(\theta)$, c.f.~Remark~\ref{rem_drop_theta}.}} $\lambda_{GD}^* = (L-\mu)/(L+\mu)$ is attained if $\alpha=2/(L+\mu)$.
        \item If $\beta\in[0,1)$, $\alpha \leq 2(1+\beta)/L$ and $\gamma>0$, then there exists $\lambda_{HB}\in[0,1)$ such that IAD+HB (fixed $K$) gives a hypergradient estimate $H_K$ satisfying
        \begin{align}
            \|h^{(K)} - D(x^{(K)}) \grad f(x^{(K)})\| \leq c(\lambda_{HB}+\gamma)^K \frac{B_{\max}}{\mu} \|\grad f(x^{(K)})\|,
        \end{align}
        for some constant $c>0$.
        The optimal rate $\lambda_{HB}^* = (\sqrt{L}-\sqrt{\mu})/(\sqrt{L}+\sqrt{\mu})$ is attained if $\alpha=4/(\sqrt{L}+\sqrt{\mu})^2$ and $\beta=(\lambda_{HB}^*)^2$.
    \end{itemize}
\end{theorem}
By comparison, the exact AD iterations \eqref{eq_hb_exact_ad} are shown to have \revision{\rerevision{sub-optimal} linear convergence rates} of size $\bigO(K\lambda^K)$ \cite[Propositions 8 \& 15]{Mehmood2020}.

We note that in their formulation as stated both IAD+GD (fixed $K$) and IAD+HB (fixed $K$) require one Hessian-vector and one Jacobian-vector product at each iteration of the hypergradient calculation (step 2). By contrast, IFT+CG only requires one Jacobian-vector product at the end of calculation, rather than one per iteration (but still one Hessian-vector product per iteration). This additional computational cost can be alleviated by reformulating the algorithms, e.g.\revision{,} as in Section~\ref{sec_unified_alg}.

In the next section we show that these inexact AD methods are actually variants of the IFT approach, and provide a unified error analysis of all three methods.

\section{Unified Hypergradient Computation Algorithm} \label{sec_unified_alg}

We revisit the inexact AD framework to realize that it is actually an approximate IFT method with specific algorithmic choices. Thus, at the end of this section we propose a unified framework that encompasses all methods in Section 2.

The next theorems make said observation for inexact AD with GD and HB.

\begin{theorem} \label{thm:gd_inexact_ad:new}
    Let $x^{(K)}$ denote the output of GD \eqref{eq_x_gd} after $K$ iterations and let $\Phi(x) = \frac 12 x^T A(x^{(K)}) x - \nabla f(x^{(K)})^T x$. Then the gradient estimate $h^{(K)}$ after $K$ iterations of the inexact AD iterations \eqref{eq_gd_inexact_ad} can also be computed as $h^{(K)} = - B(x^{(K)})^T q^{(K)}$ \revision{with} $q^{(0)} = 0$ and
    \begin{align}
	    q^{(k+1)} &= q^{(k)} - \alpha \nabla \Phi(q^{(k)}), \quad k = 0, \dots, K-1 .\label{eq_gd_inexact_ad:new}
    \end{align}
    
    \begin{proof}
    Notice that the iteration \eqref{eq_gd_inexact_ad} can be written as $h^{(K)} = - \alpha \revision{B(x^{(K)})^T} \sum_{k=0}^{K-1}  \tilde x^{(k)}$. Hence we want to show that the new iteration \eqref{eq_gd_inexact_ad:new} satisfies $q^{(K)} = \alpha \sum_{k=0}^{K-1} \tilde x^{(k)}$ which we prove by induction where we will frequently use $\nabla \Phi(q) = Aq - \nabla f(x^{(K)})$ and $\tilde x^{(0)} = \nabla f(x^{(K)})$.
    
    For $K=0$, using the initial condition and $\nabla \Phi(q^{(0)}) = -\nabla f(x^{(K)})$, \eqref{eq_gd_inexact_ad:new} implies 
    \begin{align*}
        q^{(1)} 
        &= q^{(0)} - \alpha \nabla \Phi(q^{(0)})
        = \alpha \nabla f(x^{(K)}) = \alpha \tilde x^{(0)}.
    \end{align*}

    Now, let the assertion be true for $K-1$. Due to the initial condition, an alternative way to write the \revision{iterations \eqref{eq_gd_inexact_ad}} is
    \begin{align*}
     \tilde x^{(K-1)} &= \tilde x^{(0)} - \sum_{k=0}^{K-2} \alpha A\tilde x^{(k)}.
    \end{align*}
    Thus, with the induction hypothesis  
    \begin{align*}
     \tilde x^{(K-1)} &= \nabla f(x^{(K)}) - A q^{(K-1)} = -\nabla \Phi(q^{(K-1)})
    \end{align*}
    and therefore
    \begin{align*}
        q^{(K)} 
        &= q^{(K-1)} - \alpha \nabla \Phi(q^{(K-1)})
        = \alpha \sum_{k=0}^{K-2} \tilde x^{(k)} + \alpha \tilde x^{(K-1)}
        = \alpha \sum_{k=0}^{K-1} \tilde x^{(k)}. \qedhere
    \end{align*}%
    \end{proof}
\end{theorem}

A similar observation can be made for HB.
\begin{theorem} \label{thm:hb_inexact_ad:new}
    Let $x^{(K)}$ denote the output of HB \eqref{eq_x_hb} after $K$ iterations and let $\Phi(x) = \frac 12 x^T A(x^{(K)}) x - \nabla f(x^{(K)})^T x$. Then the gradient estimate $h^{(K)}$ after $K$ iterations of the inexact AD iterations \eqref{eq_hb_inexact_ad} can also be computed as $h^{(K)} = - B(x^{(K)})^T q^{(K)}$ with $q^{(0)} = q^{(-1)} = 0$ and
    \begin{align}
	    q^{(k+1)} &= q^{(k)} - \alpha \nabla \Phi(q^{(k)}) + \beta(q^{(k)} - q^{(k-1)}), \qquad k = 0, 1, \dots, K-1.\label{eq_hb_inexact_ad:new}
    \end{align}   
    \begin{proof}
    The proof is similar to the proof of Theorem~\ref{thm:gd_inexact_ad:new}. As before, we notice that the iteration \eqref{eq_hb_inexact_ad} can be written as $h^{(K)} = - \alpha \sum_{k=0}^{K-1} B(x^{(K)})^T \tilde x^{(k)}$. Hence we want to show that the new iteration \eqref{eq_hb_inexact_ad:new} satisfies $q^{(K)} = \alpha \sum_{k=0}^{K-1} \tilde x^{(k)}$ which we prove by induction. 
    
    For $K=0$, using the initial conditions and $\nabla \Phi(q^{(0)}) = -\nabla f(x^{(K)})$, \eqref{eq_hb_inexact_ad:new} implies 
    \begin{align*}
        q^{(1)} 
        &= q^{(0)} - \alpha \nabla \Phi(q^{(0)}) + \beta(q^{(0)} - q^{(-1)})
        = \alpha \nabla f(x^{(K)}) = \alpha \tilde x^{(0)}.
    \end{align*}
    Similarly for $K=1$,
    \begin{align*}
        q^{(2)} 
        &= q^{(1)} - \alpha \nabla \Phi(q^{(1)}) + \beta(q^{(1)} - q^{(0)})
        = \alpha \tilde x^{(0)} - \alpha (A(\alpha \tilde x^{(0)}) - \nabla f(x^{(K)})) + \beta \alpha \tilde x^{(0)} \\
        &= \alpha \tilde x^{(0)} + \alpha \left(\tilde x^{(0)} - \alpha A\tilde x^{(0)} + \beta (\tilde x^{(0)} - \tilde x^{(-1)})\right)
        = \alpha \tilde x^{(0)} + \alpha \tilde x^{(1)}.
    \end{align*}

    Now, let the assertion be true for $K-1$ and $K-2$. Due to the initial conditions, an alternative way to write the \revision{iterations \eqref{eq_hb_inexact_ad_2}} is
    \begin{align*}
     \tilde x^{(K-1)} &= \tilde x^{(0)} - \sum_{k=0}^{K-2} \alpha A\tilde x^{(k)} + \beta \tilde x^{(K-2)},
    \end{align*}
    Thus,  
    \begin{align*}
        q^{(K)} 
        &= q^{(K-1)} - \alpha \nabla \Phi(q^{(K-1)}) + \beta(q^{(K-1)} - q^{(K-2)}) \\
        &= \alpha \sum_{k=0}^{K-2} \tilde x^{(k)} - \alpha (A(\alpha \sum_{k=0}^{K-2} \tilde x^{(k)}) - \nabla f(x^{(K)})) + \beta \alpha \tilde x^{(K-2)} \\
        &= \alpha \sum_{k=0}^{K-2} \tilde x^{(k)} + \alpha \left( - \sum_{k=0}^{K-2} \alpha A\tilde x^{(k)} + \tilde x^{(0)} + \beta \tilde x^{(K-2)}\right)
        = \alpha \sum_{k=0}^{K-2} \tilde x^{(k)} + \alpha \tilde x^{(K-1)}
        = \alpha \sum_{k=0}^{K-1} \tilde x^{(k)}.\qedhere
    \end{align*}
    \end{proof}
\end{theorem}

The analysis in Theorem \ref{thm:gd_inexact_ad:new} and \ref{thm:hb_inexact_ad:new} implies that $q^{(K)}$ approximately solves $A(x^{(K)}) q = \nabla f(x^{(K)})$. Thus, inexact AD is a special case of IFT with GD/HB as the lower-level solver and GD/HB to solve the system of linear equations \revision{for the upper-level hypergradient estimate}. This observation motivates us to define a generalised IFT algorithm which captures both approaches outlined in Section 2.

\begin{algorithm}[h!]
\caption{IFT to compute gradient estimate $\t h$}\label{alg:ift}
\begin{algorithmic}[1]
\State Find an approximate solution $\tilde x$ of the lower-level problem using any method of choice with any number of iterations or accuracy.
\State Find an approximate solution $\tilde q$ solving $A(\tilde x) q = \grad f(\tilde x)$  using any method of choice with any number of iterations or accuracy.
\State Return the gradient estimate $\tilde h := -B(\tilde x)^T \tilde q$.
\end{algorithmic}
\end{algorithm}

\begin{remark}
If we choose a method in step~1 to find a solution with $\varepsilon$ accuracy and CG in step 2 to find a solution with residual accuracy $\delta$, then \revision{we recover Algorithm~\ref{alg:ift+cg}} as proposed in \cite{Pedregosa2016, Zucchet2022}. If we choose GD/HB in step~1 for $K$ iterations and GD/HB in step~2 for $K$ iterations, then \revision{we get Algorithms~\ref{alg:iad+gd} and \ref{alg:iad+hb}} as proposed in \cite{Mehmood2020}. However, going forward both strategies can also be combined given more flexibility to the user.
\end{remark}

\section{A priori and \rerevision{a posteriori error analysis}} \label{sec_bounds}
We now derive new a priori and a posteriori error bounds for hypergradients as computed in Algorithm~\ref{alg:ift}. The a priori bound shows explicit convergence results that are independent of the lower-level solution estimate $\t x$.
By contrast, the a posteriori bound gives an estimate on the error that is computable (in the sense that no knowledge of $x^*$ is required). Note also, that in contrast to \cite{Pedregosa2016, Zucchet2022}~(see also Theorem \ref{thm_ift_v1} and \ref{thm_ift_v2}), we do not make any assumption on the size of the errors, nor that these estimates are part of a bilevel programming scheme.


\begin{theorem}[A posteriori bound] \label{thm:aposteriori}
Suppose Assumptions~\ref{ass_smoothness}, \ref{ass_smoothness_extra} and \ref{ass_smoothness_upper} hold.
Let $\t x, \t q$ and $\t h$ be the output of Algorithm~\ref{alg:ift}.
Let $\t \varepsilon := \|\nabla g(\t x)\|/\mu$ and $\t \delta := \|A(\t x) \t q - \nabla f(\t x)\|$. Note $\|\t x - x^*\| \leq \t \varepsilon$. Moreover, we define
$$c(x) := \frac{L_{\nabla f}\|B(x)\| }{\mu} + L_{A^{-1}} \|\grad f(x)\| \|B(x)\| +  \frac{L_B \|\grad f(x)\|}{\mu},$$
where $L_{A^{-1}}$ is the Lipschitz constant for $A(x)^{-1}$, which exists by \lemref{lem_lip_Ainv} below.
Then, the following a posteriori bound holds:
\begin{align*}
\|\t h - h^*\| 
&\leq c(\t x) \t \varepsilon + \frac{\|B(\t x)\|}{\mu} \t \delta + \frac{L_B L_{\nabla f}}{\mu} \t \varepsilon^2.
\end{align*}
\end{theorem}

\begin{theorem}[A priori bound] \label{thm:apriori}
Suppose Assumptions~\ref{ass_smoothness}, \ref{ass_smoothness_extra} and \ref{ass_smoothness_upper} hold.
Let $\t x, \t q$ and $\t h$ be the output of Algorithm~\ref{alg:ift}, computed such that $\|\t x - x^*\| \leq \varepsilon$ and $\|A(\t x) \t q - \nabla f(\t x)\| \leq \delta$. Then, with $c$ as defined in Theorem \ref{thm:aposteriori} the following a priori bound holds:
\begin{align*}
\|\t h - h^*\| 
&\leq c(x^*) \varepsilon + \frac{\|B(x^*)\|}{\mu} \delta + \frac{L_B L_{\nabla f}}{\mu} \varepsilon^2  + \frac{L_B}{\mu} \delta \varepsilon.
\end{align*}
In particular, $\|\t h - h^*\| = \bigO(\varepsilon + \delta + \varepsilon^2 + \delta^2)$ and $\|\t h - h^*\| \to 0$ as $\varepsilon, \delta \to 0$.
\end{theorem}

\subsection{Proofs of Theorems~\ref{thm:aposteriori} and \ref{thm:apriori}}
\begin{lemma} \label{lem:bound:linsys}
Assume that $\|A(\t x)^{-1}\| \leq \mu^{-1}$ for any $\t x$. Then, for any $\t x$ and $\t q$ it holds that
\begin{align*}
\|B(\t x)^T \t q - B(\t x)^T A(\t x)^{-1} \grad f(\t x)\|
&\leq \frac{\|B(\t x)\|}{\mu} \|A(\t x) \t q - \grad f(\t x)\|.
\end{align*}
\begin{proof}
A direct calculation and multiplying with $A(\t x)^{-1} A(\t x)$ it holds that
\begin{align*}
\|B(\t x)^T \t q - B(\t x)^T A(\t x)^{-1} \grad f(\t x)\| 
&\leq \|B(\t x)\| \|\t q - A(\t x)^{-1} \grad f(\t x)\| \\
&= \|B(\t x)\| \|A(\t x)^{-1} (A(\t x) \t q - \grad f(\t x))\| \\
&\leq \|B(\t x)\| \|\|A(\t x)^{-1}\| \|A(\t x) \t q - \grad f(\t x)\|
\end{align*}
The assertion then follows from $\|A(\t x)^{-1}\| \leq \mu^{-1}$.
\end{proof}
\end{lemma}

\begin{lemma} \label{lem:ld}
Let $A^{-1}$ and $B$ be Lipschitz continuous with constants $L_{A^{-1}}$ and $L_B$, respectively, and $\|A(x)^{-1}\| \leq \mu$ for any $x$. Then, the mapping $D$ is locally Lipschitz continuous. Specifically, for any $x_1, x_2$ and $L_D(x) := \|B(x)\| L_{A^{-1}} + L_B /\mu$ it holds that
\begin{align*}
\|D(x_1) - D(x_2)\|
&\leq L_D(x_1) \|x_1 - x_2\|.
\end{align*}
\begin{proof}
Adding and subtracting $B(x_1)^T A(x_2)^{-1}$ and the triangle inequality yields
\begin{align*}
\|D(x_1) - \revision{D(x_2)}\|
&= \|B(x_1)^T A(x_1)^{-1} - B(x_2)^T A^{-1}(x_2)\| \\
&\leq \|B(x_1)\| \|A(x_1)^{-1} - A(x_2)^{-1}\| + \|A(x_2)^{-1}\| \|B(x_1) - B(x_2)\|
\end{align*}
Then the result follows from Lipschitz continuity as well as the estimate $\|A(x_2)^{-1}\| \leq \mu$.
\end{proof}
\end{lemma}

\begin{lemma} \label{lem:bound:df}
Let the assumptions of Lemma~\ref{lem:ld} hold  and let $\nabla f$ be Lipschitz continuous with constant $L_{\nabla f}$. Let $c$ be as defined in Theorem~\ref{thm:aposteriori}. Then, for any $x_1, x_2$ it holds that
\begin{align*}
\|D(x_1) \grad f(x_1) - D(x_2) \grad f(x_2)\|\
&\leq c(x_1) \|x_1 - x_2\| + \frac{L_B L_{\nabla f}}{\mu} \|x_1 - x_2\|^2.
\end{align*}
\begin{proof}
We add and subtract $D(x_2) \grad f(x_1) $ and use the triangle inequality to yield
\begin{align*}
\|D(x_1) \grad f(x_1) - D(x_2) \grad f(x_2)\|
&\leq \|\grad f(x_1)\| \|D(x_1) - D(x_2)\| + \|D(x_2)\| \|\grad f(x_1) - \grad f(x_2)\|.
\end{align*}
Notice further that
\begin{align*}
\|D(x_2)\| 
= \|B(x_2)^T A(x_2)^{-1}\|
\leq \|B(x_2)\| \|A(x_2)^{-1}\|
\leq \frac{\|B(x_2)\|}{\mu}.
\end{align*}
Combining both inequalities and invoking Lemma~\ref{lem:ld} leads to
\begin{align*}
\|D(x_1) \grad f(x_1) - D(x_2) \grad f(x_2)\|
&\leq \|\grad f(x_1)\| \|D(x_1) - D(x_2)\| + \|D(x_2)\| \|\grad f(x_1) - \grad f(x_2)\| \\
&\leq \|\grad f(x_1)\| L_D(x_1) \|x_1 - x_2\| + \frac{\|B(x_2)\|}{\mu} L_{\nabla f} \|x_1 - x_2\| \\
&\leq \frac{\|B(x_1)\|}{\mu} L_{\nabla f} \|x_1 - x_2\| + \frac{L_B}{\mu} L_{\nabla f} \|x_1 - x_2\|^2 + \|\grad f(x_1)\| L_D(x_1) \|x_1 - x_2\|,
\end{align*}
where the last inequality follows from the Lipschitz continuity of $B$,
\begin{align*}
\|B(x_2)\| \leq \|B(x_1)\| + \|B(x_1) - B(x_2)\| \leq \|B(x_1)\| + L_B \|x_1 - x_2\|. \qquad \qedhere
\end{align*}
\end{proof}
\end{lemma}

The previous two Lemmas needed the Lipschitz continuity of $A^{-1}$. We show next that \revision{this} follows directly from our assumptions on $A$.

\begin{lemma}
Let $A$ be Lipschitz continuous with constant $L_A$ and $\|A(x)^{-1}\| \leq \mu$ for any $x$. Then, the mapping $A^{-1}$ is  Lipschitz continuous with constant $L_{A^{-1}} \leq L_A/\mu^2$. \label{lem_lip_Ainv}
\end{lemma}
\begin{proof}
Straightforward calculations \rerevision{lead} to
\begin{align*}
\|A(x)^{-1} - A(y)^{-1}\|
&= \revision{\|A(x)^{-1}(A(y)-A(x))A(y)^{-1}\|
\leq \|A(x)^{-1}\| \|A(y) - A(x)\|\|A(y)^{-1}\|}
\end{align*}
and the assertion follows directly from the assumptions on $A$.
\end{proof}

We are now able to prove our main results.

\begin{proof}[Proof of Theorem~\ref{thm:aposteriori}]
We start bounding the error in the hypergradient using Lemma~\ref{lem:bound:linsys}.
\begin{align}
\|\t h - h\| 
&= \|B(\t x)^T \t q - D(x^*) \grad f(x^*)\| \notag \\
&\leq \|B(\t x)^T \t q - B(\t x)^T A(\t x)^{-1} \grad f(\t x)\| + \|D(\t x) \grad f(\t x) - D(x^*) \grad f(\t x^*)\| \notag \\
&\leq \frac{\|B(\t x)\|}{\mu} \|A(\t x) \t q - \grad f(\t x)\| + \|D(\t x) \grad f(\t x) - D(x^*) \grad f(\t x^*)\| \label{eq:thm:bounds:step1}
\end{align}

For the a posteriori bound, we invoke Lemma~\ref{lem:bound:df} with $x_1 = \t x$ and $x_2 = x^*$ and apply it to \eqref{eq:thm:bounds:step1},
\begin{align*}
\|\t h - h\| 
&\leq \frac{\|B(\t x)\|}{\mu} \|A(\t x) \t q - \grad f(\t x)\| 
+ \|D(\t x) \grad f(\t x) - D(x^*) \grad f(\t x^*)\| \\
&\leq \frac{\|B(\t x)\|}{\mu} \|A(\t x) \t q - \grad f(\t x)\| 
+ c(\t x) \|\t x - x^*\| + \frac{L_B L_{\nabla f}}{\mu} \|\t x - x^*\|^2.
\end{align*}

Then using the notation $\t \varepsilon = \|\nabla g(\t x)\|/\mu$ and $\t \delta = \|A(\t x) \t q - \nabla f(\t x)\|$ we arrive at the assertion,
\begin{align*}
\|\t h - h\| 
&\leq \frac{\|B(\t x)\|}{\mu} \t \delta
+ c(\t x) \t \varepsilon + \frac{L_B L_{\nabla f}}{\mu} \t \varepsilon^2. \qedhere
\end{align*}
\end{proof}

\begin{proof}[Proof of Theorem~\ref{thm:apriori}]
As in the proof of Theorem~\ref{thm:aposteriori} we use Lemma~\ref{lem:bound:linsys} to get \eqref{eq:thm:bounds:step1}.
For the a priori bound, we then invoke Lemma~\ref{lem:bound:df} with $x_1 = x^*$ and $x_2 = \t x$ and apply it to \eqref{eq:thm:bounds:step1},
\begin{align*}
\|\t h - h\| 
&\leq \frac{\|B(\t x)\|}{\mu} \|A(\t x) \t q - \grad f(\t x)\| 
+ \|D(\t x) \grad f(\t x) - D(x^*) \grad f(\t x^*)\| \\
&\leq \frac{\|B(\t x)\|}{\mu} \|A(\t x) \t q - \grad f(\t x)\| 
+ c(x^*) \|\t x - x^*\| + \frac{L_B L_{\nabla f}}{\mu} \|\t x - x^*\|^2.
\end{align*}

Using the a priori estimates $\|A(\t x) \t q - \grad f(\t x)\| \leq \delta$ and $\|\t x - x^*\| \leq \varepsilon$ together with the Lipschitz continuity of $B$ yields
\begin{align*}
\|\t h - h\| 
&\leq \frac{\|B(\t x)\|}{\mu} \|A(\t x) \t q - \grad f(\t x)\| 
+ c(x^*) \|\t x - x^*\| + \frac{L_B L_{\nabla f}}{\mu} \|\t x - x^*\|^2 \\
&\leq \frac{\|B(\t x)\|}{\mu} \delta 
+ c(x^*) \varepsilon + \frac{L_B L_{\nabla f}}{\mu} \varepsilon^2 
\leq \frac{\|B(x^*)\|}{\mu} \delta + \frac{L_B}{\mu} \delta \varepsilon
+ c(x^*) \varepsilon + \frac{L_B L_{\nabla f}}{\mu} \varepsilon^2. \qedhere
\end{align*}
\end{proof}

\subsection{Specialized \rerevision{a priori bounds}} \label{sec_specialized_bounds}
The above framework is generic in that no specific algorithms are required for the lower-level solver and linear system solver.
Our a posteriori bounds are completely solver independent, because they use $\|\grad g(\t{x})\|$ and $\|A(\t{x})\t{q}-\grad f(\t{x})\|$  as the key error metrics, which are always available from the solver.
However our a priori bounds are solver-dependent, based on their specific convergence rates.
We now give some concrete examples of the a priori bounds for specific solver choices.

For the lower-level solver, we require $\|\t{x}-x^*\| \leq \varepsilon$.
In terms of the iteration count $k$, for gradient descent \eqref{eq_x_gd} with the optimal stepsize $\alpha=2/(L+\mu)$ we have
\cite[Lemma 6]{Mehmood2020}
\begin{align}
    \|x^{(k)}-x^*\| \leq (\lambda_{\text{GD}}^*)^k \|x^{(0)}-x^*\|, \label{eq_gd_convergence_1}
\end{align}
for all $k$, where $\lambda_{\text{GD}}^* = (L-\mu)/(L+\mu)$.
Alternatively, if we use heavy ball \eqref{eq_x_hb} with the optimal stepsize $\alpha=4/(\sqrt{L}+\sqrt{\mu})^2$ and momentum $\beta=(\lambda_{\text{HB}}^*)^2$, where $\lambda_{\text{HB}}^*:=(\sqrt{L}-\sqrt{\mu})/(\sqrt{L}+\sqrt{\mu})$, we have the following \cite[Lemma 13]{Mehmood2020}: for all $\gamma>0$, there exists $c>0$ such that
\begin{align}
    \|x^{(k)}-x^*\| \leq c(\lambda_{\text{HB}}^*+\gamma)^k \|x^{(0)}-x^*\|, \label{eq_hb_convergence_1}
\end{align}
for all $k$.
As a final example for the lower-level solver, we consider FISTA adapted for strongly convex problems, \cite[Algorithm 5]{Chambolle2016}.
Combining  \cite[Theorem 4.10]{Chambolle2016} with the identity $\|x-x^*\|^2 \leq (2/\mu)[g(x)-g(x^*)]$ (e.g.\revision{,}~\cite[Theorem 2.1.7]{Nesterov2004}) gives
\begin{align}
    \|x^{(k)}-x^*\|^2 \leq \min\left\{\left(1+\frac{\sqrt{\mu}}{\sqrt{L}}\right)(\lambda_{\text{FISTA}}^*)^k, \frac{4}{(k+1)^2}\right\} \frac{L}{\mu} \|x^{(0)}-x^*\|^2, \label{eq_fista_convegence_1}
\end{align}
where $\lambda_{\text{FISTA}}^*:=1-\sqrt{\mu/L}$.
Of these results, although heavy ball and FISTA both have an accelerated linear rate compared to gradient descent, we do have $\lambda_{\text{FISTA}}^* > \lambda_{\text{HB}}^*$, with a larger difference for well-conditioned problems.

For the linear system solver, our goal is to make $\|A(\t{x})q-\grad f(\t{x})\|$ small by minimizing $\Phi(q) = \frac{1}{2}q^T A(\t{x})q - \grad f(\t{x})^T q$.
We note that $A(\t{x})q-\grad f(\t{x}) = \grad \Phi(q)$, and $\Phi$ is $\mu$-strongly convex and has $L$-Lipschitz gradients. 
We can combine the above lower-level solver results with the identity $\mu\|q^{(k)}-q^*\| \leq \|A(\t{x})q-\grad f(\t{x})\| \leq L\|q^{(k)}-q^*\|$ where $q^* = A(\t{x})^{-1} \grad f(\t{x})$, and if we take $q^{(0)}=0$ then the initial residual is $\|\grad f(\t{x})\|$, we have
\begin{align}
\|A(\t{x})q^{(k)}-\grad f(\t{x})\| \leq \frac{L}{\mu} (\lambda_{\text{GD}}^*)^k \|\grad f(\t{x})\|,
\end{align}
for gradient descent, and for heavy ball: for all $\gamma>0$, there exists $c>0$ such that
\begin{align}
    \|A(\t{x})q^{(k)}-\grad f(\t{x})\| \leq c \frac{L}{\mu} (\lambda_{\text{HB}}^*+\gamma)^k \|\grad f(\t{x})\|. \label{eq_hb_convergence_2}
\end{align}
Lastly, we consider the a priori convergence rate of CG.
Combining the standard linear convergence rate in $\|q^{(k)}-q^*\|_{A}$ (e.g.\revision{,}~\cite[eq.~(5.36)]{Nocedal2006}) with the Rayleigh quotient inequalities $\mu\|y\|^2 \leq \|y\|_A^2 \leq L\|y\|^2$ we get the rate
\begin{align}
	\|A(\t{x}) q^{(k)} - \grad f(\t{x})\| \leq  2\frac{L^{3/2}}{\mu^{3/2}} (\lambda_{\text{HB}}^*)^k \|\grad f(\t{x})\|,
\end{align}
and we recover the same accelerated linear rate as heavy ball momentum.
These results cover the three motivating methods described in \secref{sec_background}.

We specifically note that for the linear solve step, heavy ball has the same accelerated rate as CG, but with an unknown constant, so CG is to be  preferred even without considering the extra convergence theory available for CG (e.g.\revision{,}~finite termination in exact arithmetic).

\section{Numerical \rerevision{results}} \label{sec_numerics}
We now present numerical comparisons of the different hypergradient estimation methods.
Our results have three components: in \secref{sec_numerics_quadratic} we compare the quality of the a priori and a posteriori error bounds, in \secref{sec_numerics_hypercleaning} we show how the choice of hypergradient estimation method impacts the quality of the overall optimization, and lastly in \secref{sec_numerics_icnn} we demonstrate the utility of our approach for learning high-quality neural network regularizers for image denoising.

\subsection{Quality of \rerevision{error bounds}} \label{sec_numerics_quadratic}
We first use a simple example problem to compare the quality of the a priori and a posteriori bounds derived in \secref{sec_bounds}.
The example problem is a simple linear least-squares problem taken from \cite[Section 6.1]{Li2022}:
\begin{subequations}
\begin{align}
    \min_{\theta\in\R^{10}} &\: F(\theta) := \|A_1 x^*(\theta) - b_1\|_2^2, \\
    \text{s.t.} &\: x^*(\theta) := \argmin_{x\in\R^{10}} \|A_2 x+A_3\theta-b_2\|_2^2,
\end{align}
\end{subequations}
where $A_i\in\R^{1000\times 10}$ have random  i.i.d.~entries from $\operatorname{Unif}([0,1])$, and $b_i\in\R^{1000}$ are given by $b_1=A_1\hat{x}_1+0.01y_1$ and $b_2=A_2\hat{x}_2+A_3\t{\theta}+0.01y_2$ where $\hat{x}_1$, $\hat{x}_2$ and $\t{\theta}\in\R^{10}$ have i.i.d.~$\operatorname{Unif}([0,1])$ entries and $y_1,y_2\in\R^{1000}$ are independent standard Gaussian vectors.
For our experiments we pick \revision{the test evaluation point} $\theta$ to be the vector of all ones.

Because of the simple structure of this problem, it is easy to compute $x^*(\theta)$ analytically and get all requisite Lipschitz constants.
Hence we can explicitly compute the true hypergradient $\grad F(\theta)$ and all error bounds explicitly.
The only exception is the a priori bound for HB/IAD+HB, which has the unknown constants $c$ and $\gamma$ in \eqref{eq_hb_convergence_1} and \eqref{eq_hb_convergence_2}. \revision{For illustration, we choose $c=1$ and $\gamma=0$ but there is no guarantee that this will give a true bound and such results are denoted with an asterisk in the figures below.}

In our results, we \revision{compare} the three different lower-level solvers discussed in \secref{sec_specialized_bounds}: GD~\eqref{eq_x_gd}, HB~\eqref{eq_x_hb} and FISTA (adapted for strongly convex problems as per \cite[Algorithm 5]{Chambolle2016}), all with optimal stepsize and momentum parameters.
We also use the three hypergradient methods discussed in \secref{sec_specialized_bounds}, namely CG, GD and HB.
We run the lower-level solvers for up to 100 iterations to get $\t{x} \approx x^*$ (except for \figref{fig_quadratic_ad_comparison} where $\t{x}=x^*$ is used), and the hypergradient solvers for up to 200 iterations.

\begin{figure}[tb]
  \centering
  \begin{subfigure}[b]{0.45\textwidth}
		\includegraphics[width=\textwidth]{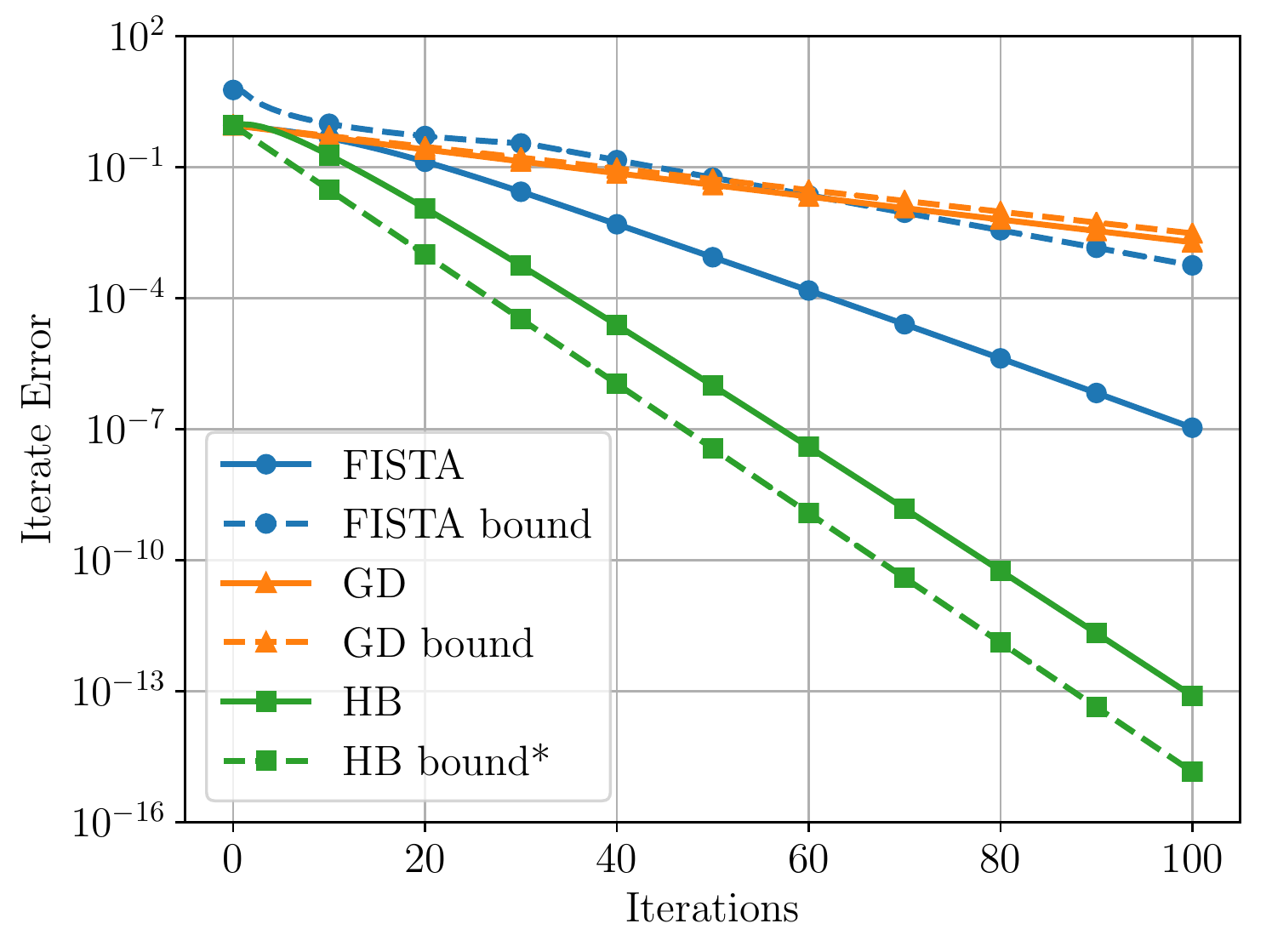}
		\caption{Lower-level solve, a priori bounds.}
	\end{subfigure}
	~
	\begin{subfigure}[b]{0.45\textwidth}
		\includegraphics[width=\textwidth]{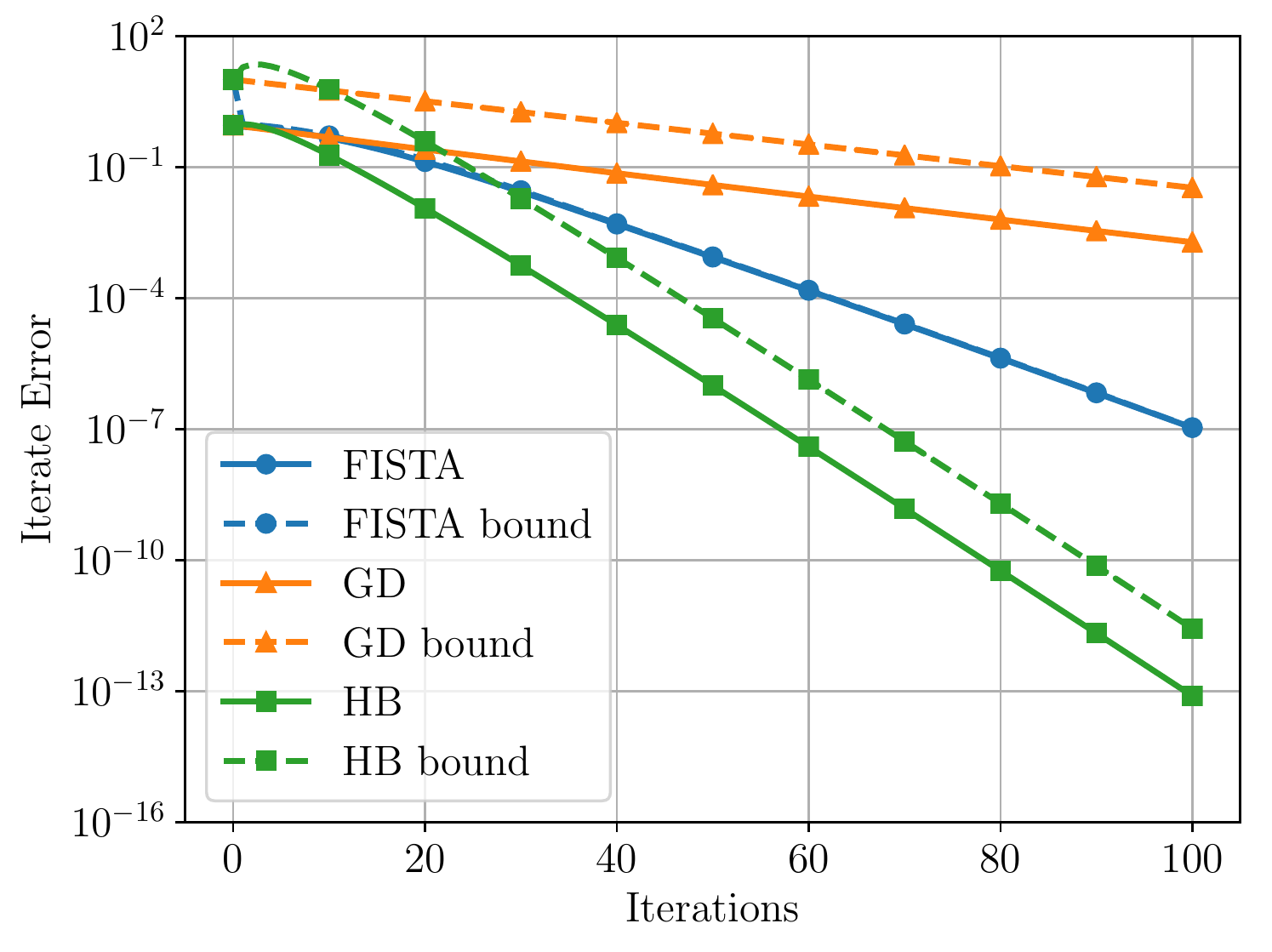}
		\caption{Lower-level solve, a posteriori bounds.}
	\end{subfigure}
	\\
	\begin{subfigure}[b]{0.45\textwidth}
		\includegraphics[width=\textwidth]{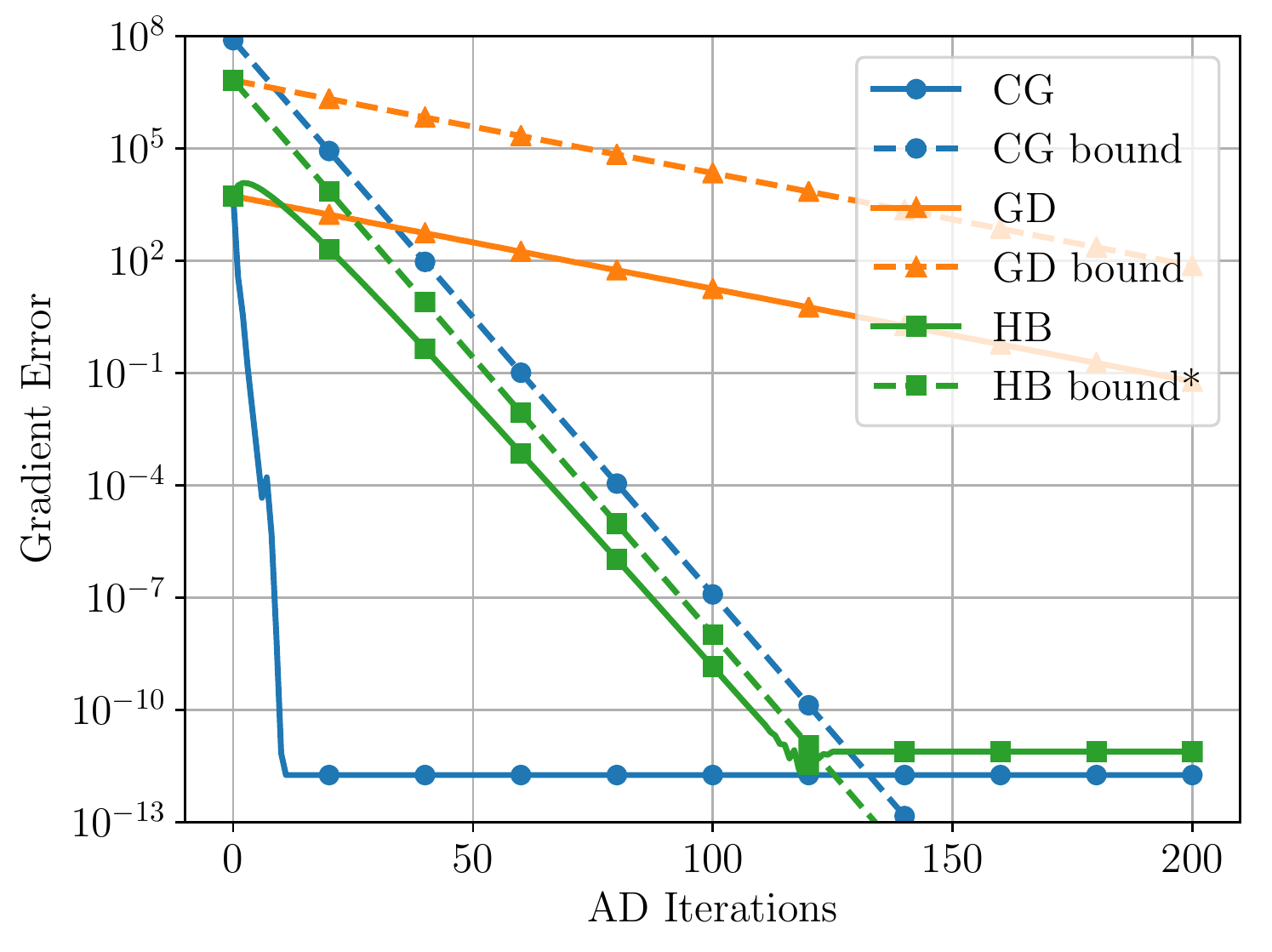}
		\caption{AD comparison, a priori bounds (exact $\t{x}=x^*(\theta)$)}
	\end{subfigure}
	~
	\begin{subfigure}[b]{0.45\textwidth}
		\includegraphics[width=\textwidth]{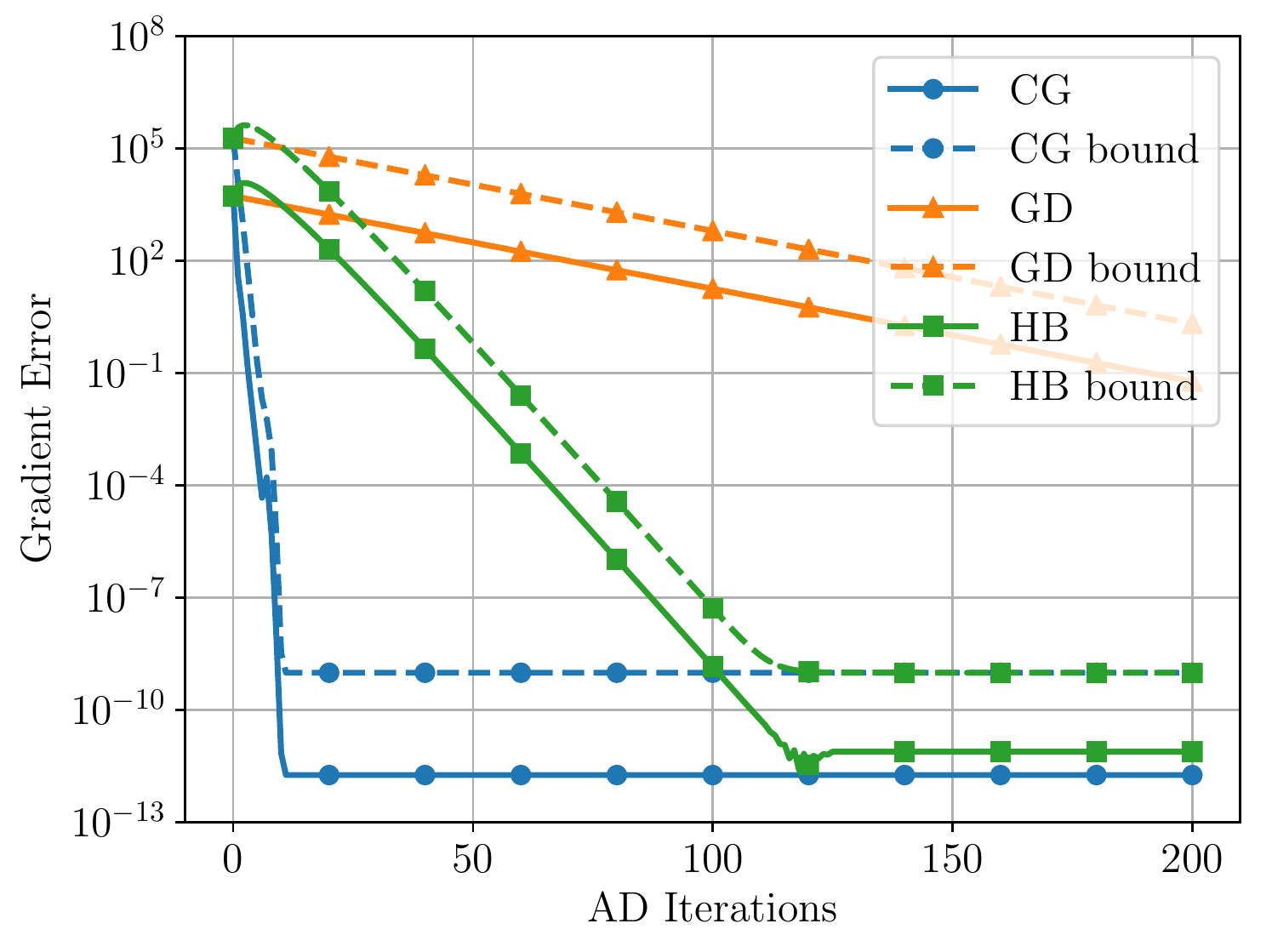}
		\caption{AD comparison, a posteriori bounds (exact $\t{x}=x^*(\theta)$)}
	\end{subfigure}
  \caption{Simple quadratic AD comparison. *The a priori bound for heavy ball uses $c=1$ and $\gamma=0$ in \eqref{eq_hb_convergence_1} for (a,b) and \eqref{eq_hb_convergence_2} for (c,d), but in reality these constants are not known and so there is no guarantee that this will actually be an upper bound on the error. \revision{Note: in (b), the FISTA bound is very tight and is almost on top of the true FISTA error.}}
  \label{fig_quadratic_ad_comparison}
\end{figure}

Firstly, \rerevision{Figures~\ref{fig_quadratic_ad_comparison}(a,b) show} the a priori and a posteriori bounds on the lower-level solvers, where the a priori bounds are from the standard linear convergence rates for the lower-level solvers (i.e.~\eqref{eq_gd_convergence_1}, \eqref{eq_hb_convergence_1} and \eqref{eq_fista_convegence_1} for GD, HB and FISTA respectively) and the a posteriori results use $\|\t{x}-x^*\| \leq \|\grad g(\t{x})\|/\mu$.
As in \cite[Figure 3]{Ehrhardt2021}, we find that the a posteriori bounds are much tighter for FISTA (and HB given that the a priori bounds are uncomputable), although the a priori bounds are better for the slowest method, GD.

\rerevision{Figures~\ref{fig_quadratic_ad_comparison}(c,d) then show} the a priori and a posteriori bounds on the hypergradient estimates, using the exact value $\t{x}=x^*$ (i.e.~$\varepsilon=0$).
We see that the a posteriori bounds are significantly tighter for CG and GD (and are the only valid option for HB).
Furthermore, the a priori bounds are not always valid once the hypergradient error is small enough that rounding errors are significant, whereas the a posteriori bounds automatically handle this.

We also consider the same results as \rerevision{Figures~\ref{fig_quadratic_ad_comparison}(c,d)}, but where $\t{x} \neq x^*$ (i.e.~$\varepsilon>0$).
These results are shown in \figref{fig_quadratic_ad_comparison_by_epsilon} in Appendix~\ref{sec_extra_numerics}.
Specifically, the fastest lower-level solver (HB) was run for $N\in\{20,60,100\}$ iterations and $\t{x}$ was taken as the final iterate, corresponding to $\varepsilon\in\{1.1\times 10^{-2}, 3.9\times 10^{-8}, 7.7\times 10^{-14}\}$ respectively.
Here, we see that the a priori bounds are tighter for large $\varepsilon$, but the a posteriori bounds become more useful as $\varepsilon\to 0$.

Lastly, \figref{fig_quadratic_ad_comparison_by_epsilon_2} shows the true hypergradient errors from \figref{fig_quadratic_ad_comparison_by_epsilon} in Appendix~\ref{sec_extra_numerics}, but comparing the overall gradient error against total computational work (measured as the sum of lower-level iterations and hypergradient iterations), for different levels of lower-level solve accuracy.
Here, we are interested in considering how to allocate a given budget of computational resources between producing more accurate lower-level solves and more accurate hypergradients.
We see that there is a genuine trade-off that must be considered: larger $N$ for more accurate lower-level solves ultimately can give significantly more accurate hypergradients, but for very small budgets a smaller $N$ should be used to allow the hypergradient iteration to run for sufficiently long.
The trade-off that appears here aligns with the necessary balance between $\varepsilon$ and $\delta$ inherent in our a priori bound (\thmref{thm:apriori}).

\begin{figure}[tb]
  \centering
	\begin{subfigure}[b]{0.45\textwidth}
		\includegraphics[width=\textwidth]{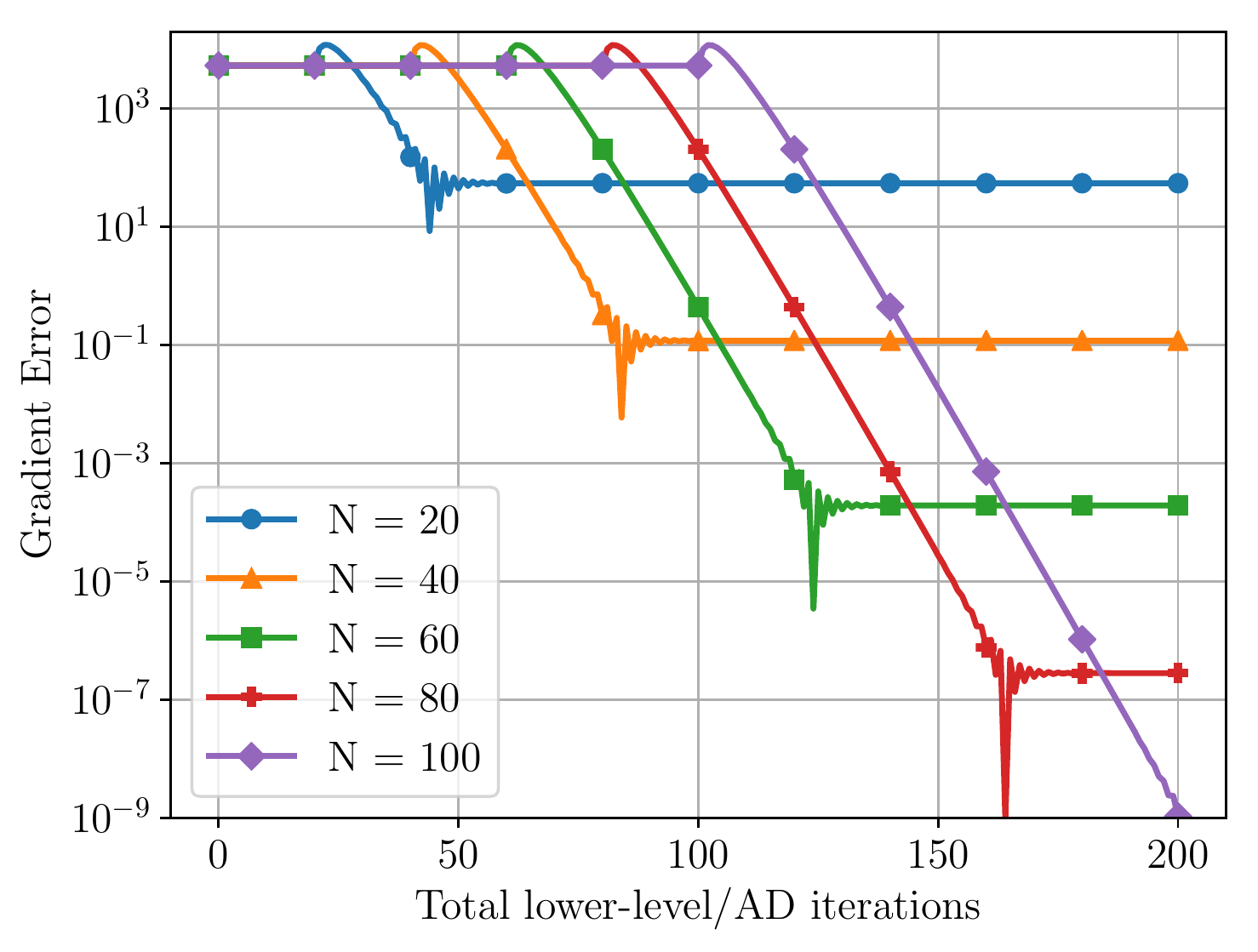}
		\caption{Using HB for AD method}
	\end{subfigure}
	~
	\begin{subfigure}[b]{0.45\textwidth}
		\includegraphics[width=\textwidth]{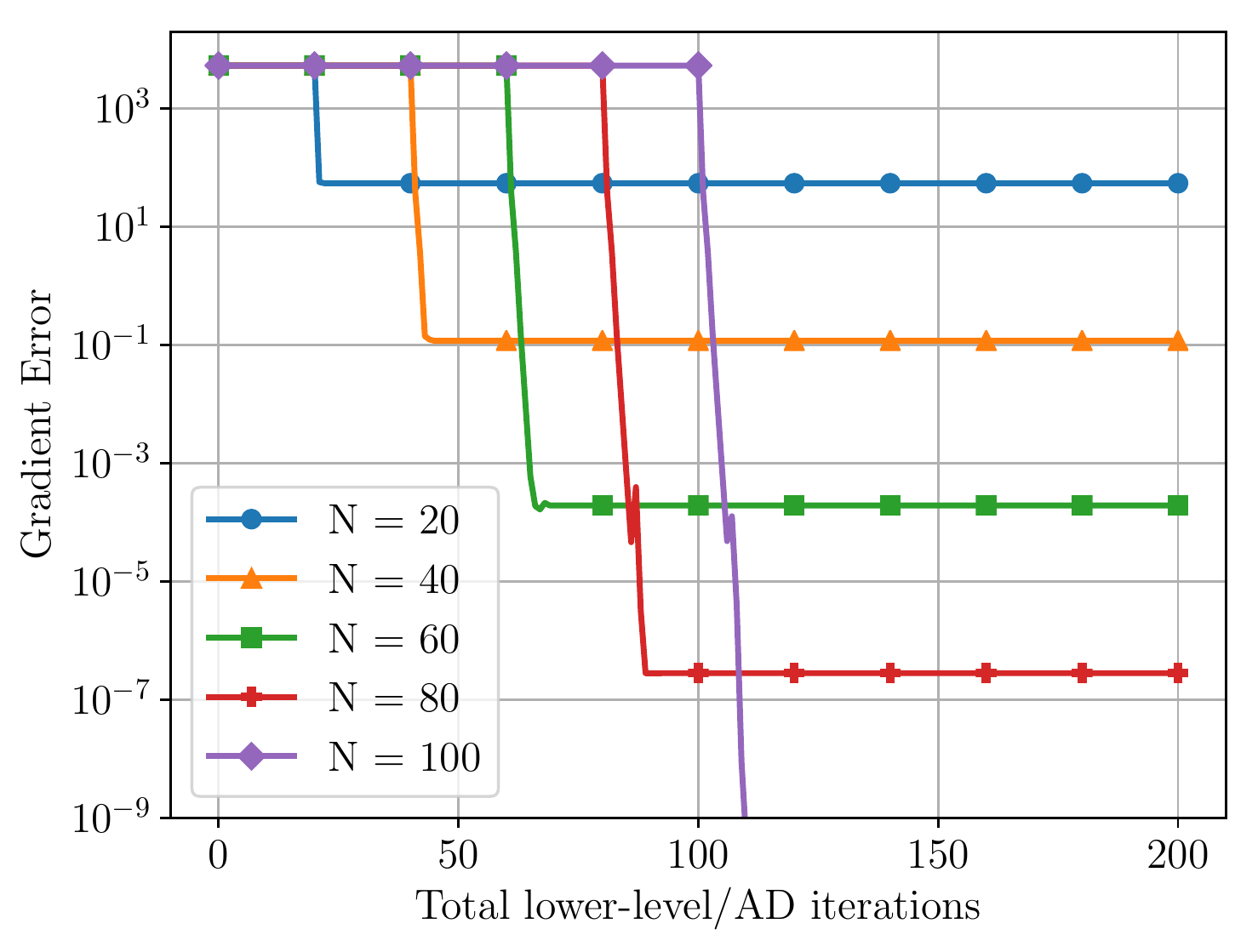}
		\caption{Using CG for AD method}
	\end{subfigure}
  \caption{Simple quadratic problem: comparing actual gradient error versus total computational work (lower-level solve plus hypergradient iterations) for different accuracies of lower-level solve ($N$ is the number of heavy ball iterations used to compute $\t{x}$).} \label{fig_quadratic_ad_comparison_by_epsilon_2}
\end{figure}

\subsection{Impact of Hypergradient Method on Optimization Quality} \label{sec_numerics_hypercleaning}
We now consider a more realistic example problem to answer the question: how does the choice of hypergradient method affect the quality of the overall bilevel optimization process?
To answer this question we use a data hypercleaning problem from \cite[Appendix B]{Yang2021}.
This process is to learn weights for all training examples in a supervised learning problem, where some training examples have corrupted labels (and so the standard equal weighting is not ideal), by minimizing loss over a validation dataset.
In this case, we consider multi-class logistic regression on MNIST with a cross-entropy loss:
\begin{subequations}
\begin{align}
    \min_{\theta\in\R^{N_{\text{train}}}} & \: F(\theta) := \frac{1}{N_{\text{val}}} \sum_{i=1}^{N_{\text{val}}} \ell(X^*(\theta) x_i^{\text{val}}, y_i^{\text{val}}), \\
    \text{s.t.} \quad & \: X^*(\theta) := \argmin_{X\in\R^{n_c \times d}} \frac{1}{N_{\text{train}}} \sum_{j=1}^{N_{\text{train}}}\sigma(\theta_j) \ell(X x_j^{\text{train}}, y_j^{\text{train}}) + C\|X\|_F^2,
\end{align}
\end{subequations}
where $\ell(y_{\text{est}},y_{\text{true}}): \R^{n_c} \times \R^{n_c} \to \R$ is the cross-entropy loss, and $\sigma(\cdot)$ is the sigmoid function.
We have $n_c=10$ classes, feature size $d=785$, $\ell_2$ penalty $C=0.001$ \revision{(as chosen in \cite{Yang2021})} and dataset sizes $N_{\text{train}}=20000$ and $N_{\text{val}} = 5000$.
A randomly chosen 10\% of the training labels $y_j^{\text{train}}$ are corrupted by choosing an incorrect label uniformly at random.
The goal of the hypercleaning problem is effectively to encourage the lower-level weights $\sigma(\theta_j)\to 0$ where $y_j^{\text{train}}$ is corrupted and $\sigma(\theta_j)\to 1$ otherwise.

As in \secref{sec_numerics_quadratic}, we use GD, HB and FISTA as lower-level solvers and CG, GD and HB as hypergradient algorithms.
To solve the full bilevel problem, we run gradient descent with constant step-size (in this case taking $\alpha=10$) on the upper-level problem using the calculated inexact hypergradients.
\revision{We use warm restarts for the lower-level solver, choosing $x^{(0)}$ in \eqref{eq_x_gd} or \eqref{eq_x_hb} (for example) to be the final value found in the previous iteration (with the previous value of $\theta$).}
Since we are interested in the impact on the full upper-level solve, we show how the upper-level objective $F(\theta)$ decreases as a function of total computational work, taken as the sum of the total lower-level iterations and hypergradient iterations.
We use this measure as each iteration of these  requires one lower-level gradient and one lower-level Hessian-vector product respectively (with a similar cost).\footnote{We ignore the contribution of  Jacobian-vector products in the hypergradient calculation, since there is only one per calculation compared to one Hessian-vector product per iteration of the hypergradient calculations.}

Our results are shown in \figref{fig_hypercleaning}, for the different choices $\varepsilon,\delta\in\{10^{-2},10^{-1}\}$.
We omit the results using GD as a hypergradient algorithm since they are all significantly worse than HB and CG (although we do show results with GD as a lower-level solver).

Comparing lines of the same \revision{shade} (i.e.~same lower-level solver), it is clear that the choice of hypergradient method has a substantial impact on the speed of the overall optimization.
Indeed, our results suggest that the choice of hypergradient algorithm is at least as important as the choice of lower-level solver.
In \figref{fig_hypercleaning}(c,d), it is even the case that using GD as a lower-level solver with CG for hypergradients outperforms using HB for both (i.e.~a non-accelerated lower-level solver with CG can outperform using an accelerated solver for both steps).

In \figref{fig_hypercleaning}(b,d), we see that the fastest solver (HB lower-level/CG hypergradients) plateaus after sufficient time.
This is because the solver has reached a level of accuracy where the first hypergradient iteration $q^{(0)}=0$ has a sufficiently small residual and so a zero hypergradient is returned.
However this is not a fundamental limit: the level of $F(\theta)$ corresponding to the plateau is exceeded in \figref{fig_hypercleaning}(c) by taking a smaller value of $\delta$.
This suggests that a dynamic upper-level algorithm where $\varepsilon$ and $\delta$ are carefully decreased to zero may be a superior method (c.f.~the fixed decrease schedule for the bilevel solver HOAG  \cite{Pedregosa2016}).
The development and analysis of such an approach is delegated to future work.



\begin{figure}[tb]
  \centering
  \begin{subfigure}[b]{0.45\textwidth}
		\includegraphics[width=\textwidth]{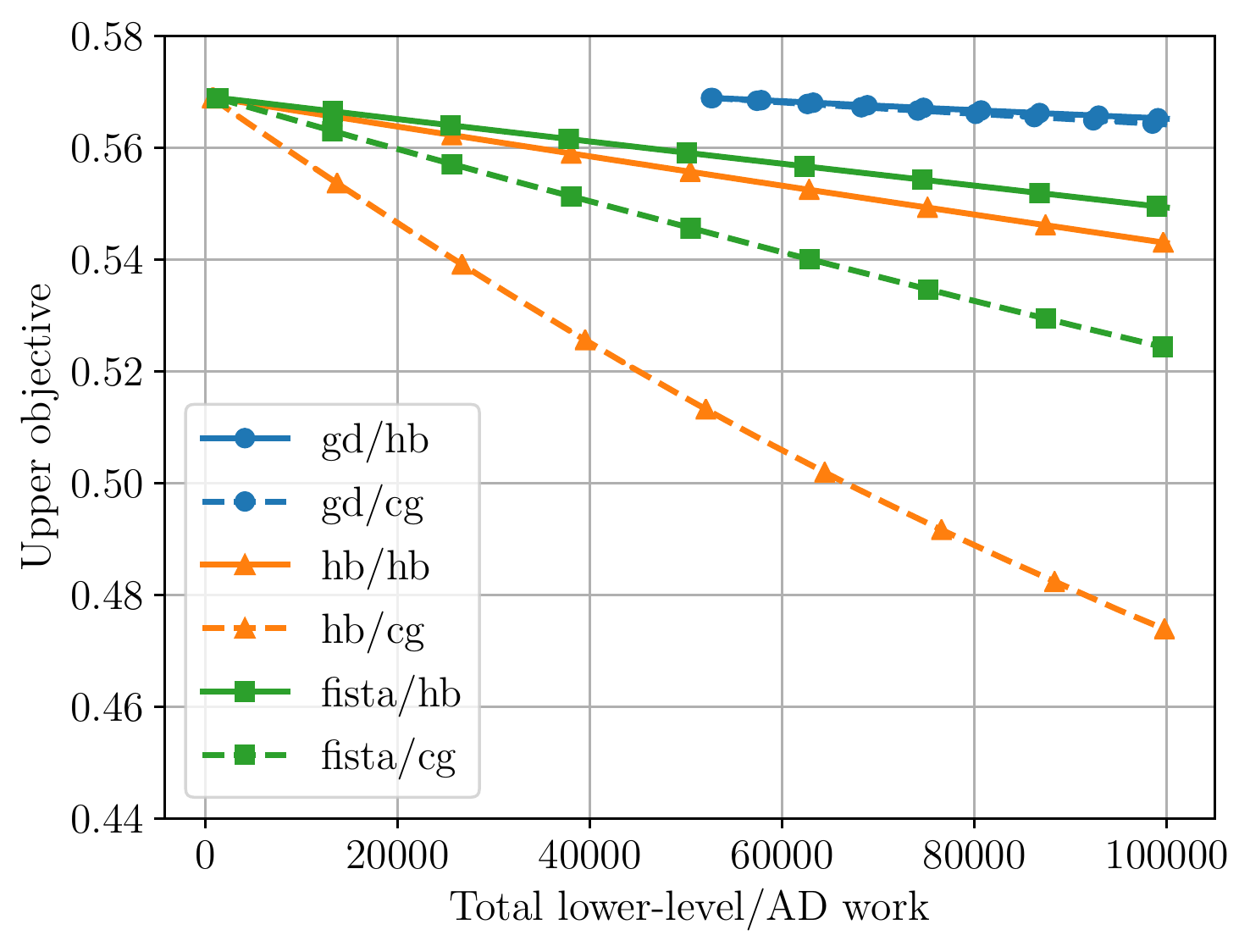}
		\caption{$\varepsilon=0.01$, $\delta=0.01$}
	\end{subfigure}
	~
	\begin{subfigure}[b]{0.45\textwidth}
		\includegraphics[width=\textwidth]{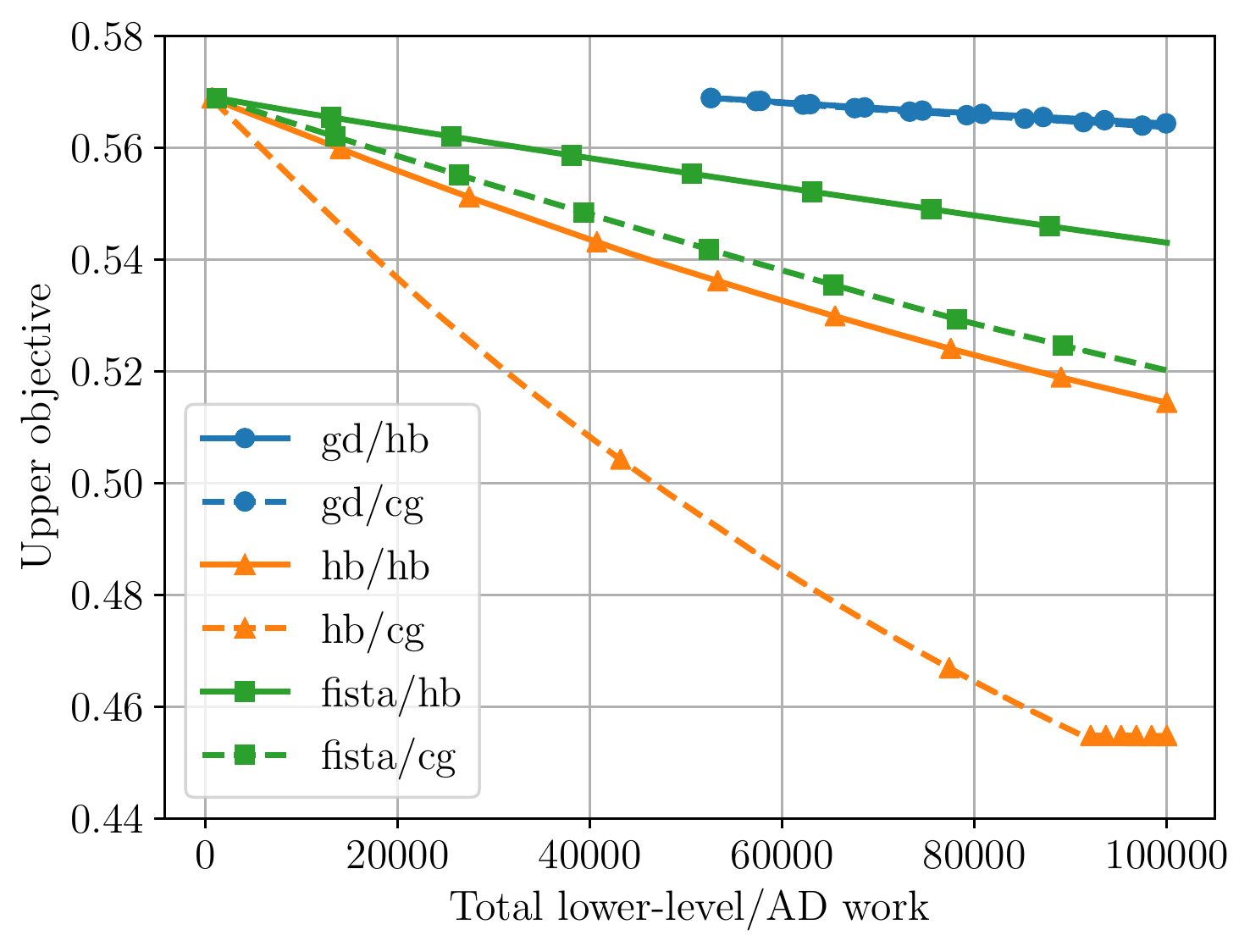}
		\caption{$\varepsilon=0.01$, $\delta=0.1$}
	\end{subfigure}
	\\
	\begin{subfigure}[b]{0.45\textwidth}
		\includegraphics[width=\textwidth]{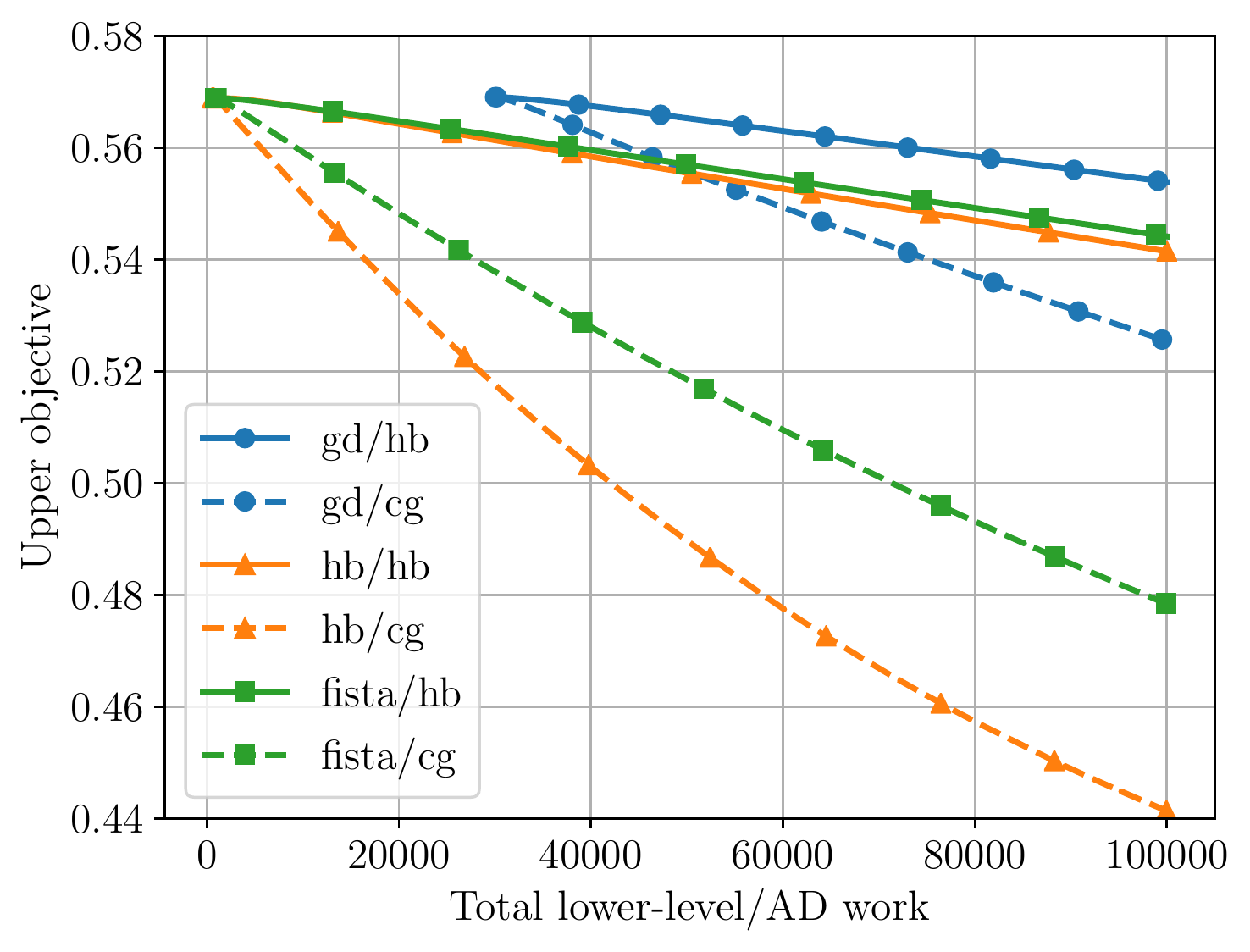}
		\caption{$\varepsilon=0.1$, $\delta=0.01$}
	\end{subfigure}
	~
	\begin{subfigure}[b]{0.45\textwidth}
		\includegraphics[width=\textwidth]{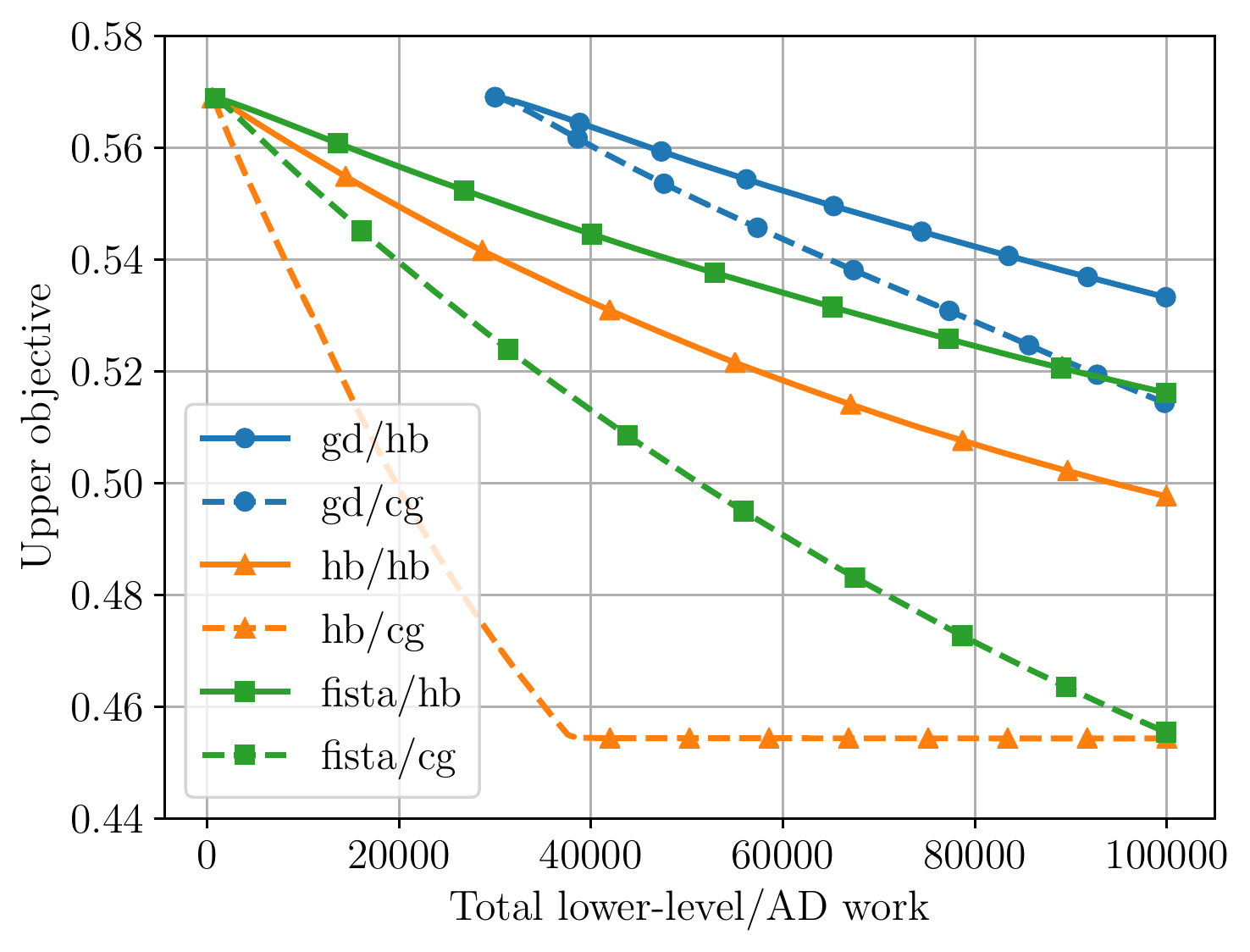}
		\caption{$\varepsilon=0.1$, $\delta=0.1$}
	\end{subfigure}
  \caption{Data hypercleaning results: upper-level objective achieved for a given amount of total work (cumulative lower-level iterations plus AD iterations) for different combinations of lower-level solver and AD method.}
  \label{fig_hypercleaning}
\end{figure}


\subsection{Quality of Learned Regularizer} \label{sec_numerics_icnn}
Lastly, we include an example demonstrating that our approach is capable of learning interesting regularizers for image denoising that outperform standard methods.

Here, our test problem is image denoising using variational regularization where the regularizer is an input-convex neural network (ICNN) \cite{Amos2017inputconvex, Mukherjee2020inputconvex}. 
Thus, the lower-level problems $g_i$ \eqref{eq_x_theta} have the form 
$g_i(y,\theta) := \|y-z_i\|^2 + R(y,\theta).$
In more detail, we take $R(y,\theta)$ to be a linear combination of an input-convex neural network and an $\ell_2$ penalty, i.e.
\begin{align}
    R(y,\theta) = \revision{\log(1+e^{\theta_1})} S(y,\theta_{3:\text{end}}) + \revision{\log(1+e^{\theta_2})} \|y\|^2. \label{eq_nn_reg}
\end{align}
The map $S$ is a shallow neural network which comprises of a single convolutional layer (with kernel length 3 and two output channels) with a softplus activation function, followed by an averaging output layer.
\revision{The use of the softplus function $t\mapsto \log(1+e^t)$ in  \eqref{eq_nn_reg} is to ensure non-negativity of the coefficients and thus convexity of $R$ in $y$.}
With this architecture, we have $n=8$ parameters $\theta$ to learn. Our implementation of this regularizer is based on \cite{Mukherjee2020inputconvex}.\footnote{\url{https://github.com/Subhadip-1/data_driven_convex_regularization}}

We use a least-squares loss as the upper-level objective \eqref{eq_upper_level}, i.e. $f_i(y) := \|y-x_i\|^2$,
where we have training data pairs $x_i,z_i\in\R^d$, corresponding to ground truth images $x_i$ and noisy images $z_i$.
Throughout, our dataset comprises true images $x_i\in\R^{d}$ for $d=64$ which are discontinuous and piecewise linear with 4 segments, each with a random slope generated from $\operatorname{Unif}([-10,10])$ and with a jump of size $\operatorname{Unif}([-1,1])$ between segments.
The noisy data $z_i$ come from perturbing $x_i$ with independent component-wise noise drawn from $\operatorname{Unif}([-\delta,\delta])$ with $\delta = \tfrac {0.2} d \sum_{j=1}^{d} |x_j|$.
To reflect a realistic imaging setting where full uncorrupted data acquisition (i.e.~collecting suitable $x_i$) is generally difficult, we consider a setting where we have $m=10$ training and 10 test images.
An example image may be seen in \figref{fig_test2_direct_compare}.

We will compare the ICNN \eqref{eq_nn_reg} with the classical total variation (TV) regularizer $R(y,\theta) = \theta \operatorname{TV}(y)$ as implemented in \cite{Ehrhardt2021}. The hyperparameter to be learned is the regularization weight $\theta>0$ and $\operatorname{TV}(y) := \sum_j \|\hat{\grad} y_j\|$ is the discretized total variation, i.e.~$\hat{\grad} y_j$ is a forward difference approximation to the gradient of $y$ at pixel $j$.
This regularizer is built around the prior that $y$ is approximately piecewise constant.

We now demonstrate that the considered bilevel learning framework with the optimized gradient estimates can produce high-quality learned input-convex neural net regularizers for image denoising.

We ran the same bilevel solver (gradient descent) as in \secref{sec_numerics_hypercleaning}, now with a constant upper-level stepsize of 0.01.
All hypergradients were calculated using FISTA as the lower-level solver and CG to solve the IFT linear system with $\varepsilon=\delta \in\{10^{-2}, 10^{-3}, 10^{-4}\}$ constant.
All solvers were run for 6 days.

\begin{figure}[tb]
  \centering
  \includegraphics[width=0.45\textwidth]{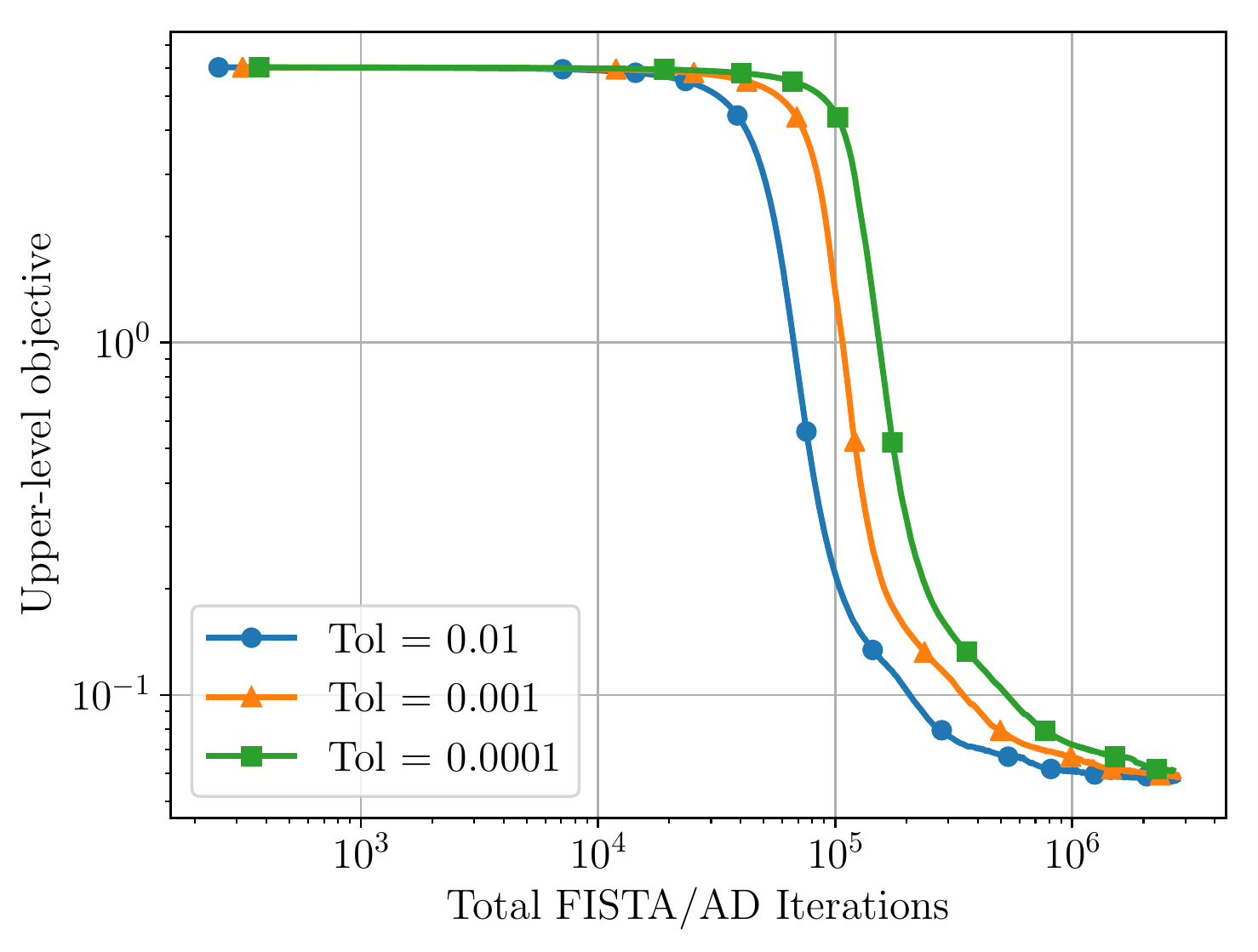}~
  \includegraphics[width=0.45\textwidth]{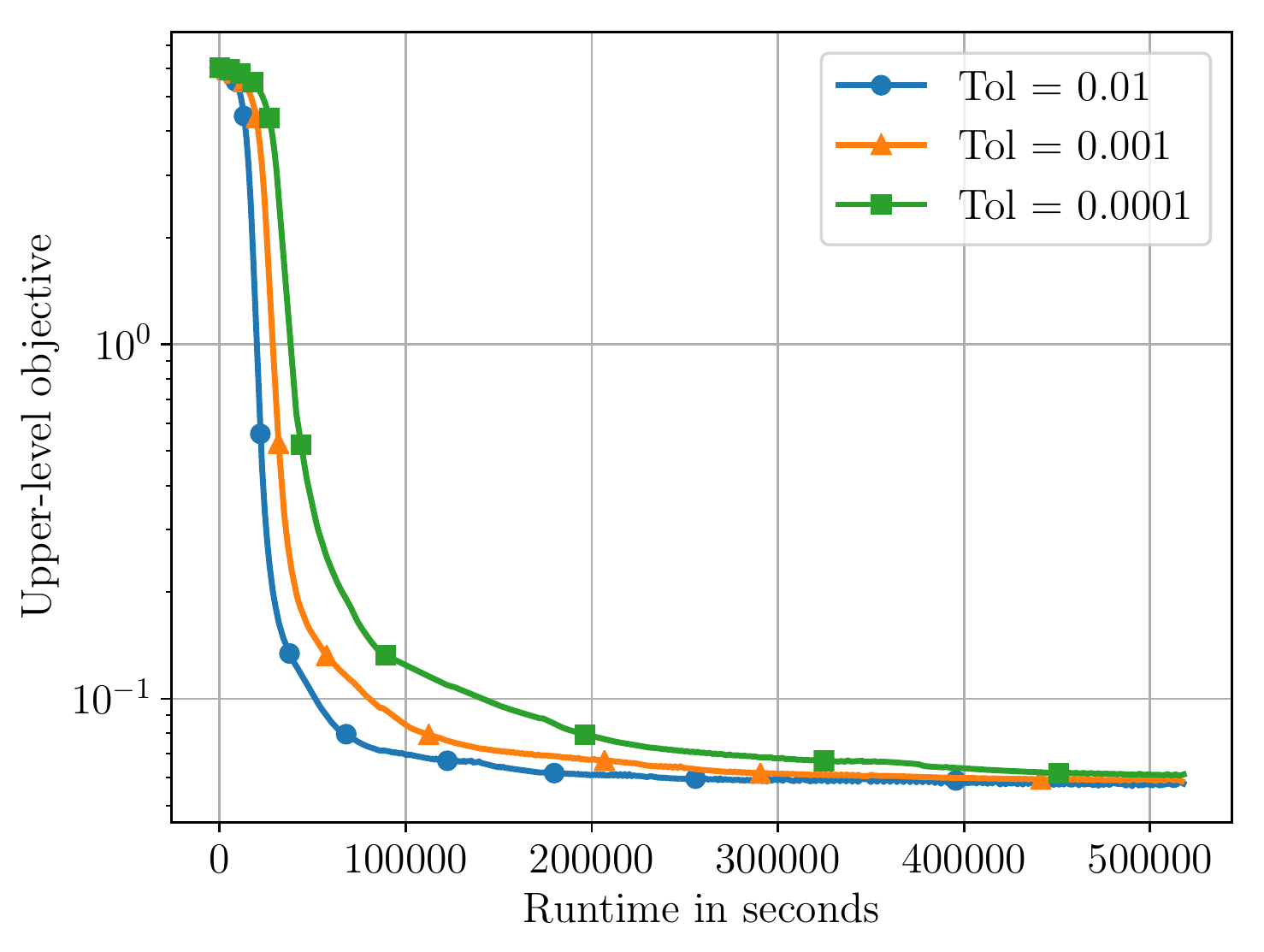}
  \caption{Convergence of bilevel solver with tolerance $\varepsilon=\delta$ constant. It can be seen that computing hypergradients with higher accuracy does not lead to an overall efficient algorithm.}
  \label{fig_bilevel_tv_obj}
\end{figure}

The resulting upper-level objective decrease (versus computational budget) is shown in \figref{fig_bilevel_tv_obj}.
Separately, we also use the derivative-free approach of \cite{Ehrhardt2021} to tune the TV regularizer weight ($\theta\approx 0.026$ being optimal for this training dataset).
Using the best learned hyperparameters $\theta$ from each method, the resulting training and test losses are shown in \tabref{tab_bilevel_tv}. We see that the proposed bilevel framework with $\varepsilon=\delta=10^{-2}$ can produce an improved training loss and test loss than a highly tuned version of the custom-designed TV regularizer.
By comparison, the input-convex neural net regularizer had no a prior information about the dataset.
\revision{We conclude by noting that we would expect the training loss to be strictly decreasing as $\varepsilon=\delta$ is reduced, which is not observed in \tabref{tab_bilevel_tv}. This can be attributed to our finite computational budget, and shows the practical trade-off between high accuracy and practical efficiency.}

\begin{table}
  \caption{Quantitative comparison between TV and ICNN. Training and test losses of tuned regularizers from different frameworks. ICNN with lowest accuracy hypergradients lead to the smallest test loss.}\vspace*{3mm}
  \label{tab_bilevel_tv}
  \centering
  \begin{tabular}{cccc}
    \hline
    Method     & Training Accuracy & Training loss     & Test loss \\
    \hline
    TV & & 0.0601  & 0.0558     \\[1mm]
    ICNN & $\varepsilon=\delta=10^{-2}$     & \textbf{0.0567} & \textbf{0.0552}      \\
    & $\varepsilon=\delta=10^{-3}$     & 0.0586       & 0.0569\\
    & $\varepsilon=\delta=10^{-4}$     & 0.0607       & 0.0582  \\
    \hline
  \end{tabular}
\end{table}

Interestingly, TV and ICNN with optimized parameters have very different properties.
In \figref{fig_test2_direct_compare} we show the two reconstructions for an example image from the test dataset with similar upper-level losses.
Clearly the TV regularizer yields reconstructions which have a strong piecewise constant preference, but the ICNN regularizer produces much smoother reconstructions with occasional jumps.

\begin{figure}[tb]
  \centering
  \begin{subfigure}[b]{0.45\textwidth}
		\includegraphics[width=\textwidth]{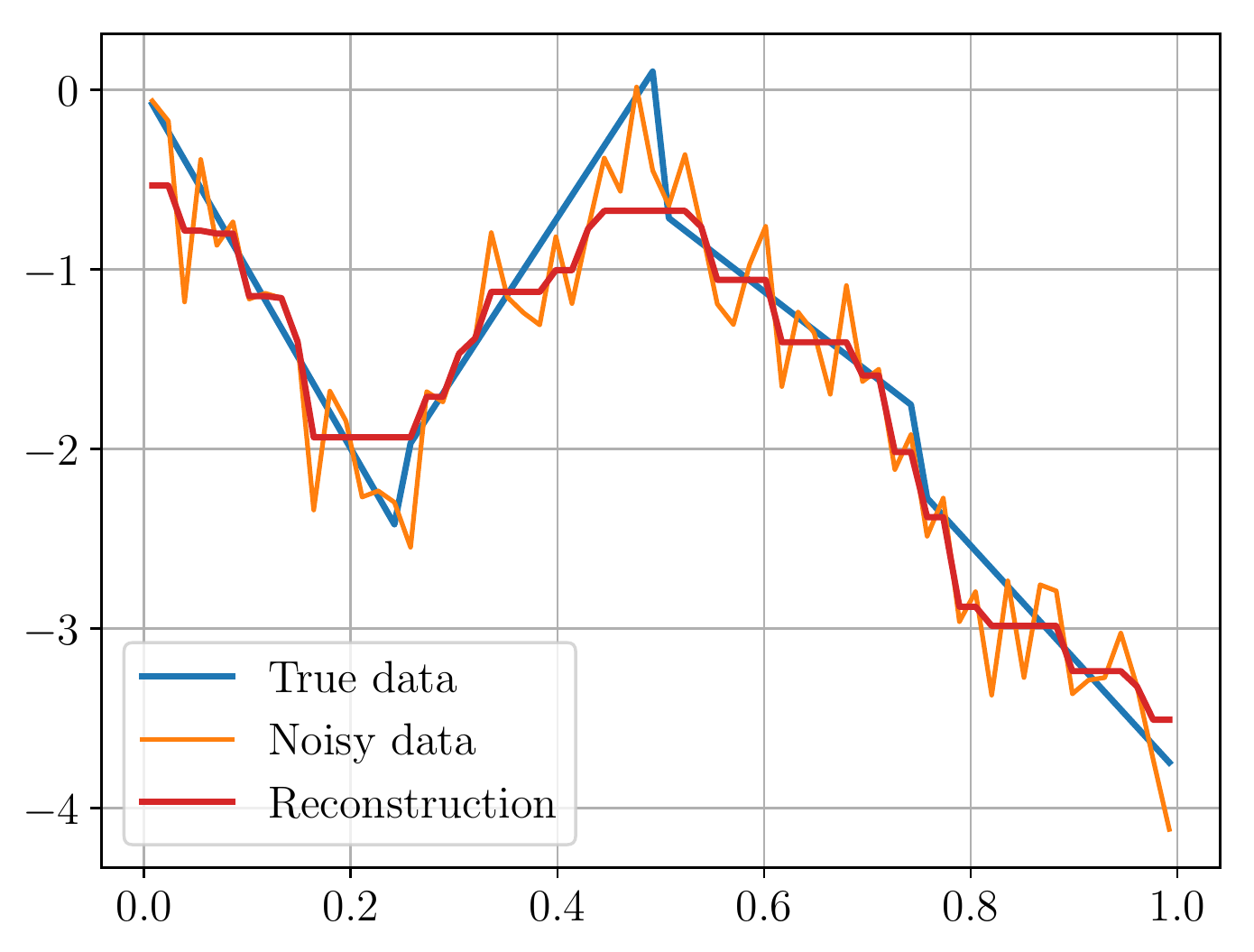}
		\caption{TV (loss 0.0601)}
	\end{subfigure}
	~
	\begin{subfigure}[b]{0.45\textwidth}
		\includegraphics[width=\textwidth]{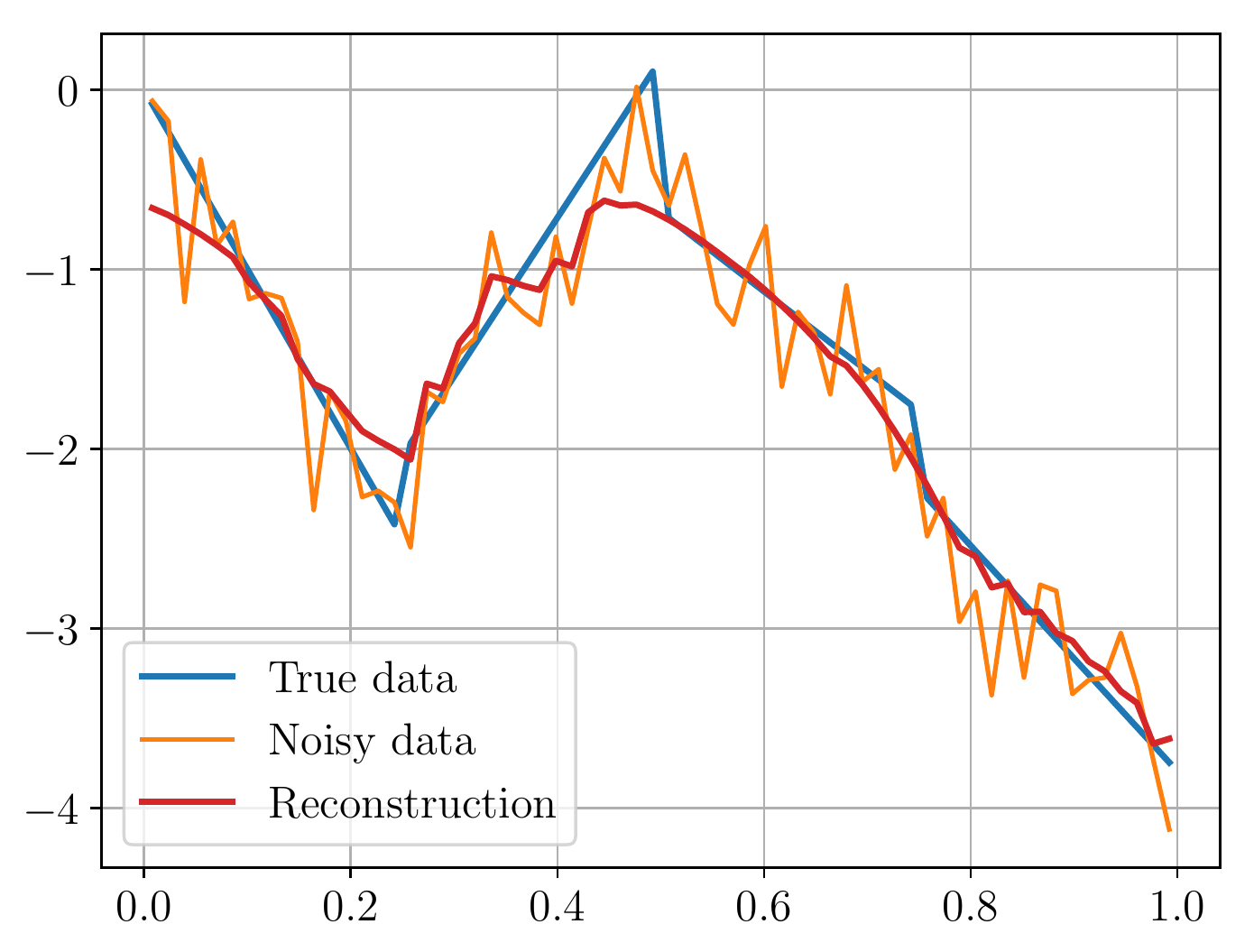}
		\caption{ICNN $\varepsilon=\delta=10^{-2}$ (loss 0.0497)}
	\end{subfigure}
  \caption{Qualitative comparison between TV and ICNN regularization using final tuned parameters. While the general approximation capabilities of the two models are similar it is worth pointing out that the learned ICNN does not lead the the staircasing artefact apparent in the TV reconstruction.}
  \label{fig_test2_direct_compare}
\end{figure}

\section{Conclusion}
This work has demonstrated that the promising inexact AD approach for computing hypergradients is equivalent to using the implicit function theorem.
This leads to a simple, unified framework for approximating hypergradients.
Our framework is flexible, with no specific requirements on solver choices and termination conditions, and is accompanied by standard linear convergence rates as well as new, computable a posteriori error bounds.
In practice these a posteriori bounds are also typically more accurate.
Importantly, our results also demonstrate that careful selection of the hypergradient approximation method is as important for bilevel optimization as choosing a lower-level solver.
Our results provide a promising foundation for building practical and rigorous bilevel optimization methods which will be addressed in future work.

\section*{Acknowledgments}

This work is supported in part by funds from EPSRC (EP/S026045/1, EP/T026693/1, EP/V026259/1) and the Leverhulme Trust (ECF-2019-478).

\bibliographystyle{unsrt}
\bibliography{refs_ima}

\appendix

\clearpage

\section{Extra Numerical Results} \label{sec_extra_numerics}
\figref{fig_quadratic_ad_comparison_by_epsilon} below shows the same results as \figref{fig_quadratic_ad_comparison}(c,d), but where $\t{x}$ is computed inexactly using $N$ iterations of heavy ball.

\begin{figure}[H]
  \centering
  \begin{subfigure}[b]{0.4\textwidth}
		\includegraphics[width=\textwidth]{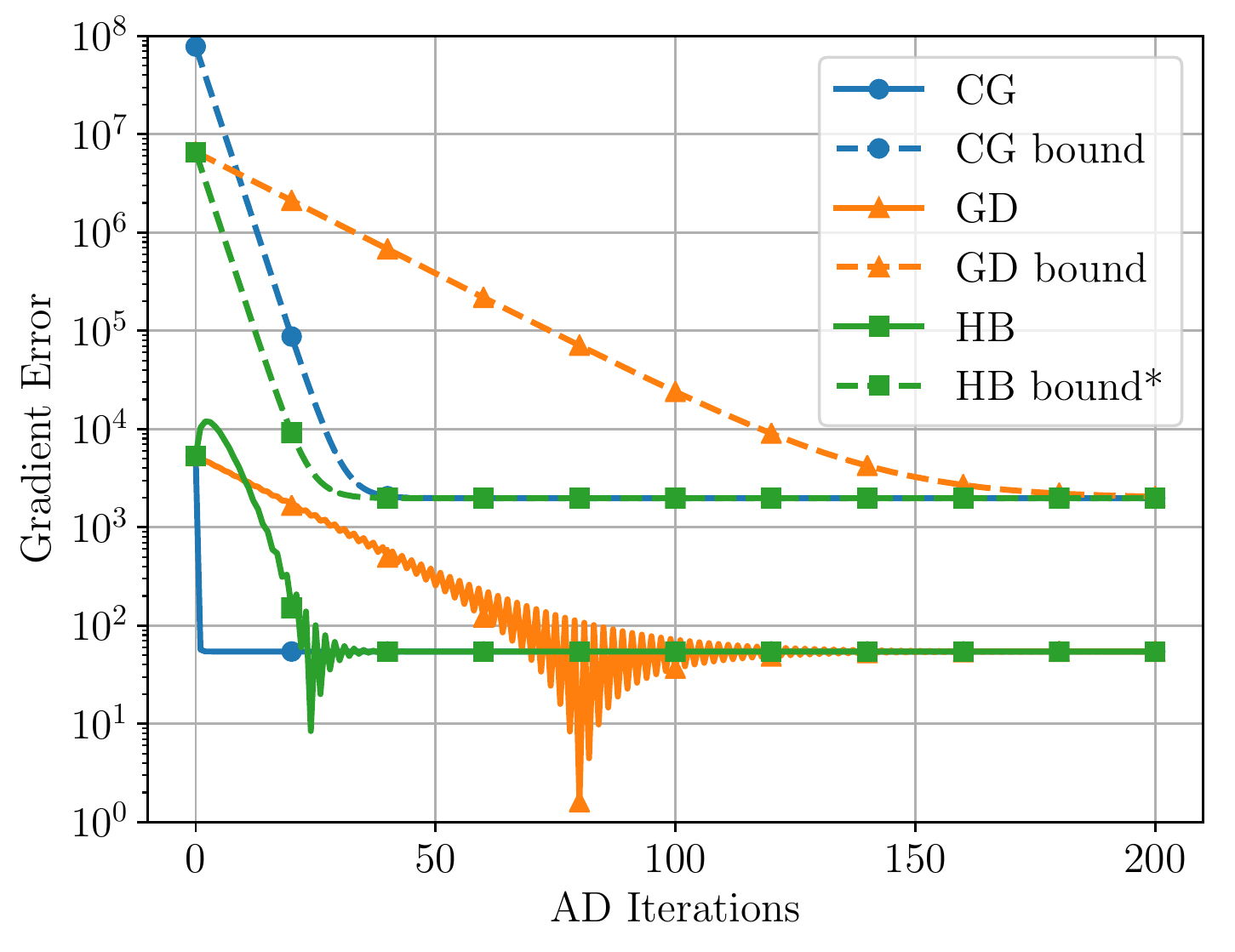}
		\caption{$N=20$ ($\varepsilon=1.1\times 10^{-2}$), a priori bounds}
	\end{subfigure}
	~
	\begin{subfigure}[b]{0.4\textwidth}
		\includegraphics[width=\textwidth]{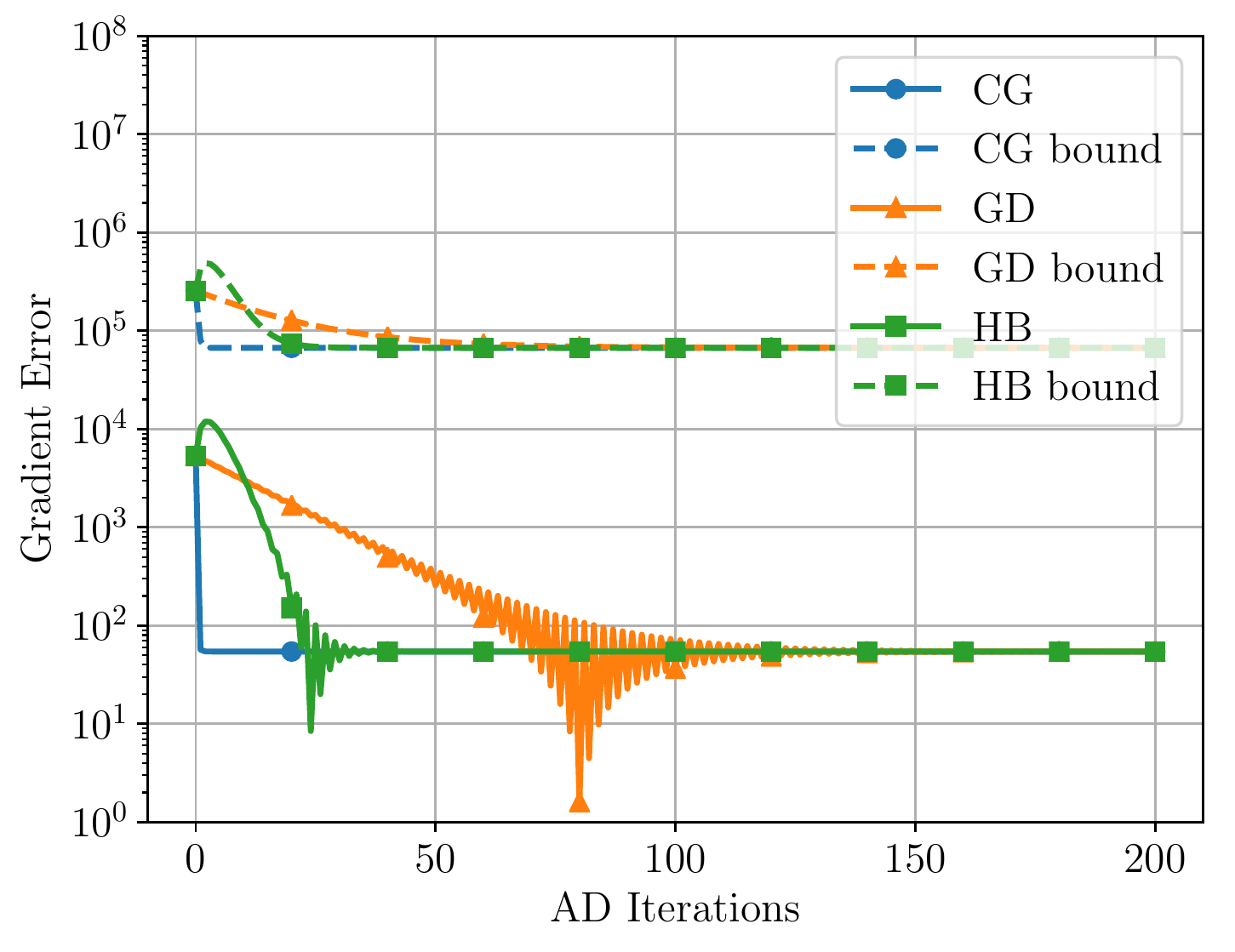}
		\caption{$N=20$ ($\varepsilon=1.1\times 10^{-2}$), a posteriori bounds}
	\end{subfigure}
	\\
	\begin{subfigure}[b]{0.4\textwidth}
		\includegraphics[width=\textwidth]{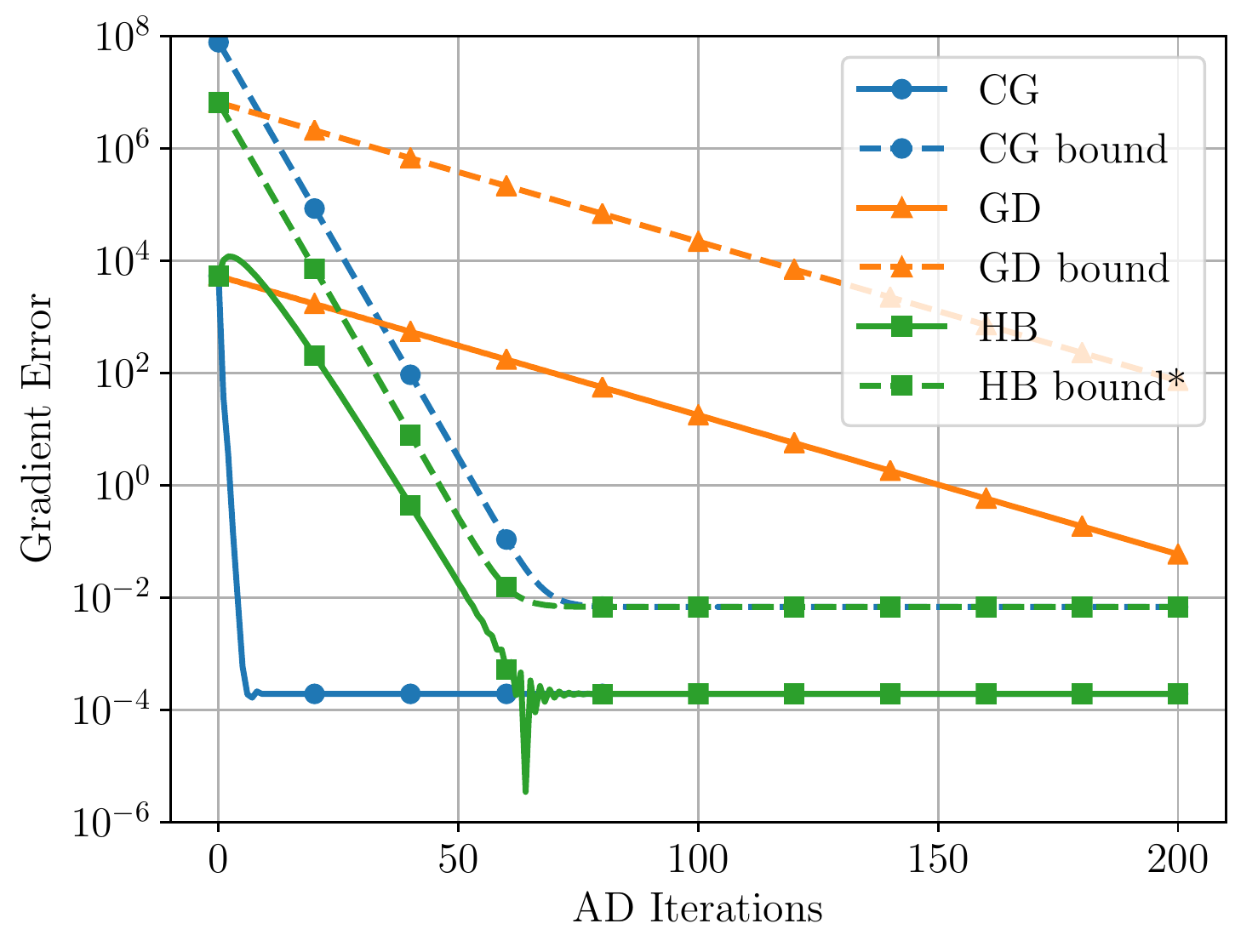}
		\caption{$N=60$ ($\varepsilon=3.9\times 10^{-8}$), a priori bounds}
	\end{subfigure}
	~
	\begin{subfigure}[b]{0.4\textwidth}
		\includegraphics[width=\textwidth]{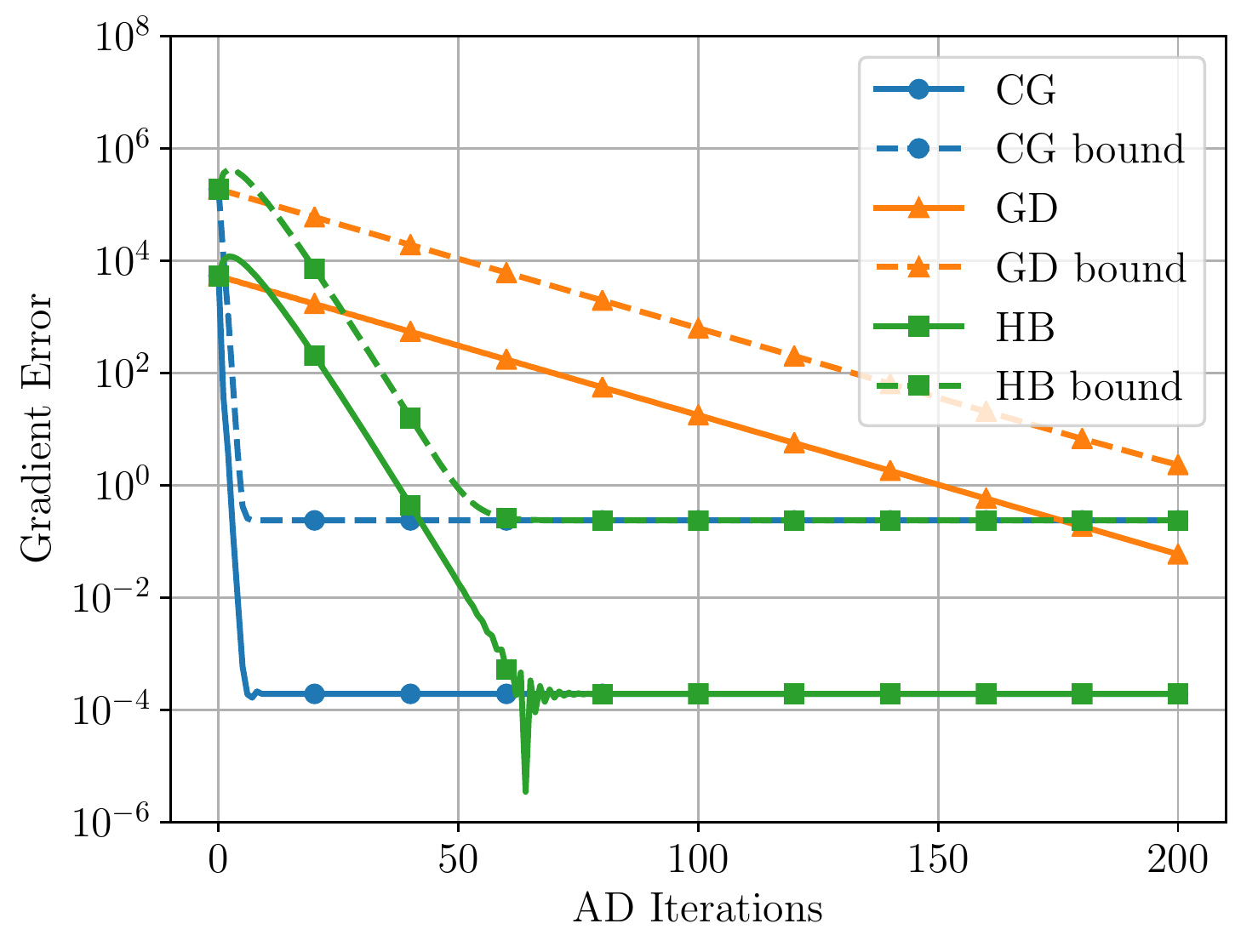}
		\caption{$N=60$ ($\varepsilon=3.9\times 10^{-8}$), a posteriori bounds}
	\end{subfigure}
	\\
	\begin{subfigure}[b]{0.4\textwidth}
		\includegraphics[width=\textwidth]{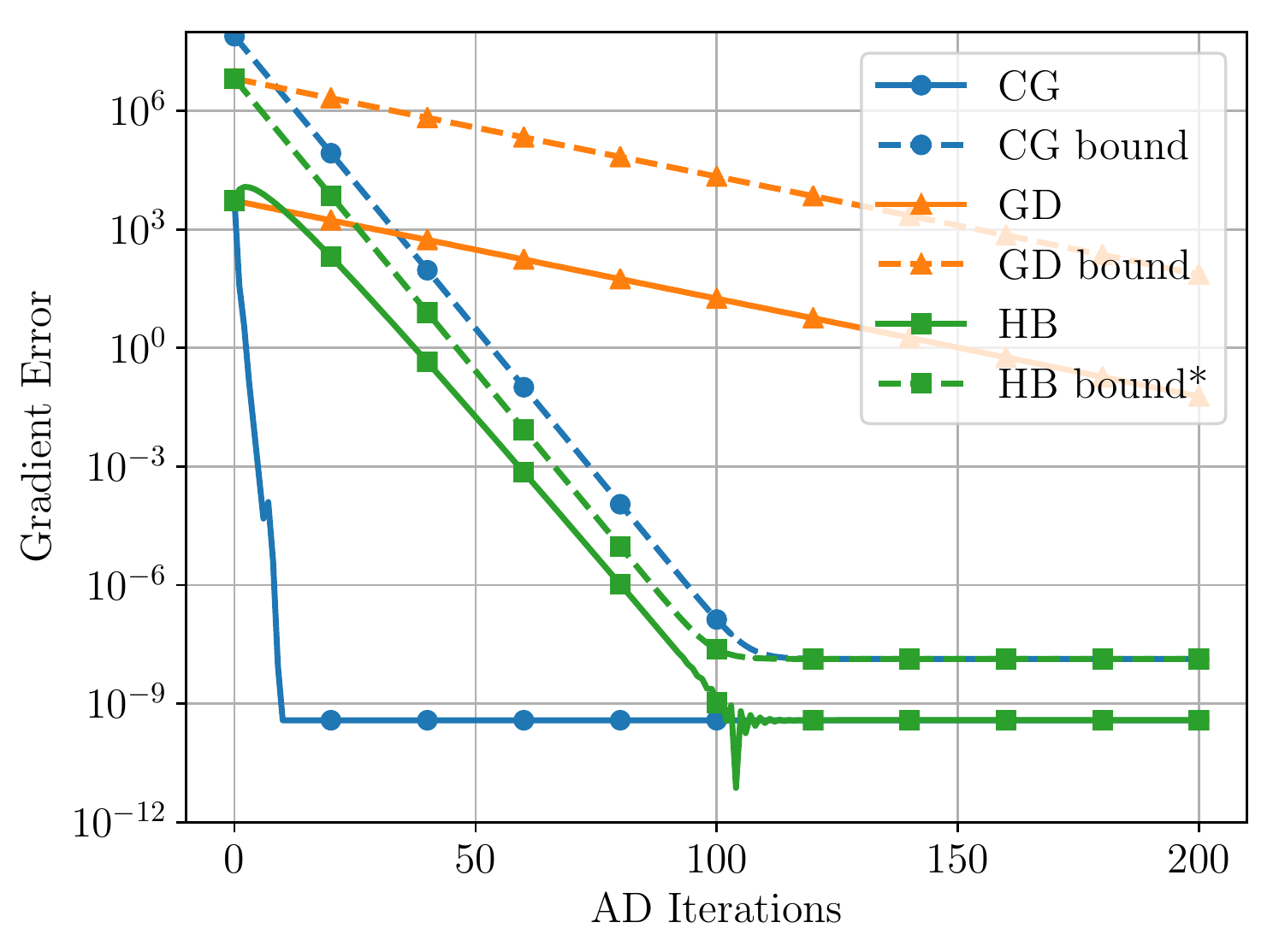}
		\caption{$N=100$ ($\varepsilon=7.7\times 10^{-14}$), a priori bounds}
	\end{subfigure}
	~
	\begin{subfigure}[b]{0.4\textwidth}
		\includegraphics[width=\textwidth]{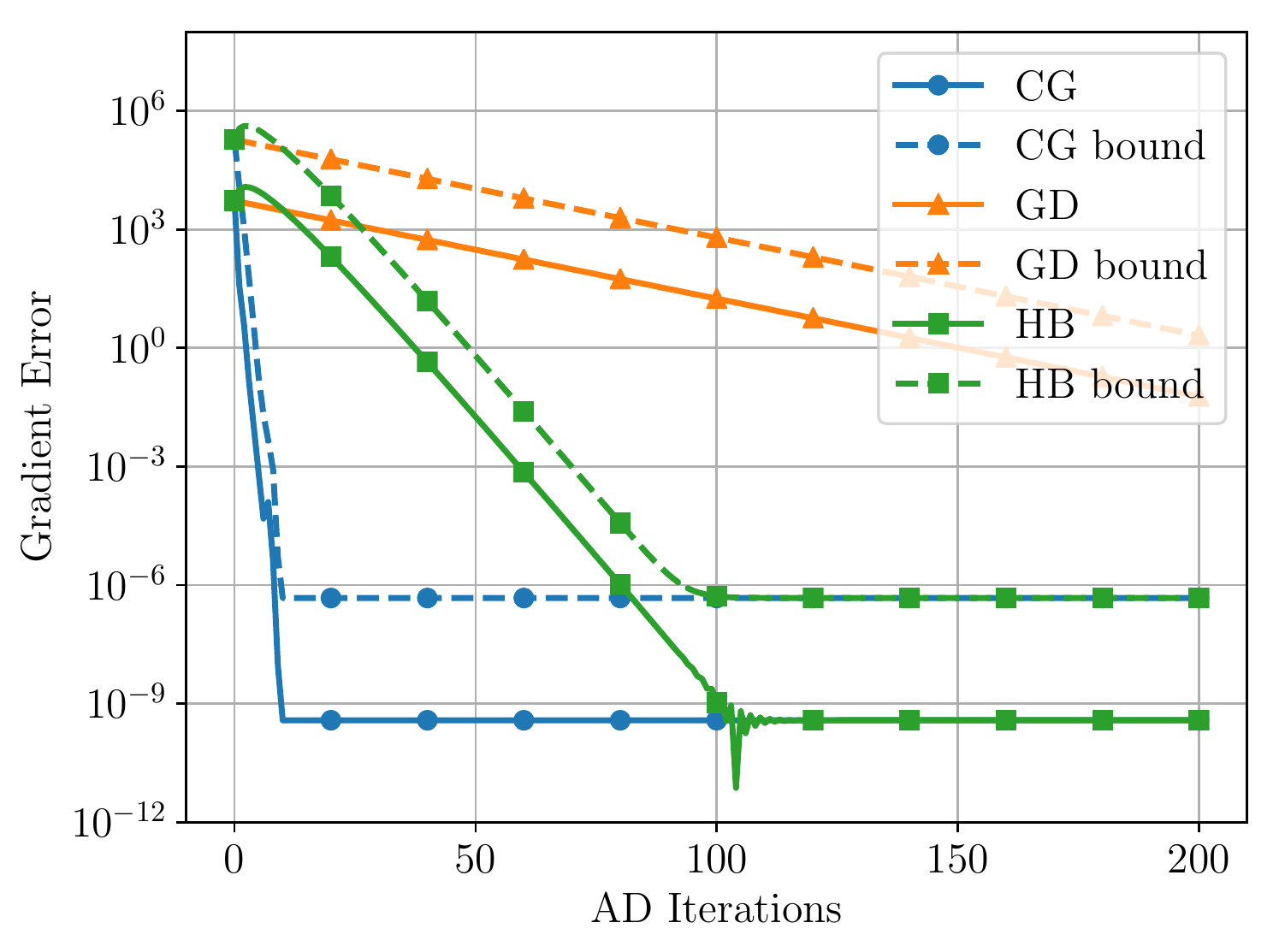}
		\caption{$N=100$ ($\varepsilon=7.7\times 10^{-14}$), a posteriori bounds}
	\end{subfigure}
  \caption{Simple quadratic AD comparison for different accuracy levels of lower-level solve ($\t{x}$ from $N$ iterations of heavy ball, yielding $\varepsilon$ as stated). *The a priori bound for heavy ball uses $c=1$ and $\gamma=0$ in \eqref{eq_hb_convergence_2}, but in reality these constants are not known and so there is no guarantee that this will actually be an upper bound on the error.}
  \label{fig_quadratic_ad_comparison_by_epsilon}
\end{figure}

\end{document}